\documentclass[a4paper,11pt]{article}
\usepackage{amsthm, amsmath, amssymb, enumerate, paralist, calc, verbatim, tocloft,amssymb}
\usepackage{tikz}
\newcommand{\llangle}{\langle\!\langle}
\newcommand{\rrangle}{\rangle\!\rangle}

\newcommand{\diagnode}[1]{\fill #1 circle (.1);}
\newcommand{\distorbit}[1]{\draw #1 circle (.2);}
\newcommand{\distlongorbit}[1]{\draw (#1 - .2, .5) -- (#1 - .2, -.5) arc (-180:0:.2) -- (#1 + .2, -.5) -- (#1 + .2, .5) arc (0:180:.2) -- cycle;}

\newtheorem{theorem}{Theorem}[section]

\newtheorem{lemma}[theorem]{Lemma}
\newtheorem{proposition}[theorem]{Proposition}

{\theoremstyle{definition}
\newtheorem{remark}[theorem]{Remark}
\newtheorem{definition}[theorem]{Definition}
\newtheorem{notation}[theorem]{Notation}
\newtheorem{construction}[theorem]{Construction}

\newtheorem*{construction*}{Construction}
}

\numberwithin{equation}{section}

\newcommand{\Z}{\mathbb{Z}} 

\newcommand{\Hh}{\mathbb{H}} 
\newcommand{\A}{\mathcal{A}}
\newcommand{\Ss}{\mathcal{S}}
 
\newcommand{\bt}{\bullet} 
\newcommand{\Xt}{\widetilde{X}} 
\newcommand{\xt}{\widetilde{x}} 
\newcommand{\minuszero}{\setminus\{0\}} 
\newcommand{\al}{\alpha} 

\newcommand\half{{\tfrac{1}{2}}}
\newcommand{\subhalf}{{\leavevmode \raise.5ex\hbox{\the\scriptfont0 1}\kern-.13em /\kern-.07em\lower.25ex\hbox{\the\scriptfont0 2}}}

\newcommand{\lsk}{{\mbox{\boldmath$\pmb($}}}
\newcommand{\rsk}{{\mbox{\boldmath$\pmb)$}}}

\newcommand{\ga}{\gamma}
\newcommand{\la}{\lambda}
\newcommand{\si}{\sigma}

\newcommand{\chacha}{characteristic}

\newcommand{\ii}{\mathbf{i}}
\newcommand{\jj}{\mathbf{j}}
\newcommand{\kk}{\mathbf{k}}

\DeclareMathOperator{\Char}{char}
\DeclareMathOperator{\Gal}{Gal}
\DeclareMathOperator{\End}{End}

\DeclareMathOperator{\Fix}{Fix}

\DeclareMathOperator{\cha}{char}

\begin{document}

\title{A new construction of Moufang quadrangles of type $E_6, E_7$ and $E_8$}
\author{Lien Boelaert%
\quad Tom De Medts%
}
\date{\today}
\maketitle

\begin{abstract}
	In the classification of Moufang polygons by J.~Tits and R.~Weiss, the most intricate case is by far the case of the exceptional Moufang quadrangles of type $E_6$, $E_7$ and $E_8$, and in fact, the construction that they present is ad-hoc and lacking a deeper explanation.
	We will show how tensor products of two composition algebras can be used to construct these Moufang quadrangles in characteristic different from 2.

	As a byproduct, we will obtain a method to construct {\em any} Moufang quadrangle
	in characteristic different from two from a module for a Jordan algebra.
\end{abstract}

\vspace*{2ex} \noindent
{\tt MSC-2010}: primary: 17A75, 17A40, 17C40, 20G15, 20G41; secondary: 17C27, 51E12.

\noindent
{\tt keywords}: Moufang polygons, Moufang quadrangles, composition algebras, octonion algebras,
 quadrangular algebras, Jordan algebras, structurable algebras, $J$-ternary algebras,
linear algebraic groups, exceptional groups, $E_6$, $E_7$, $E_8$

\newpage
\tableofcontents
\newpage

\section{Introduction}

In the late sixties, Jacques Tits introduced an (at that time) innovative tool to study semisimple linear algebraic groups of
positive relative rank, namely the theory of spherical buildings.
Especially in the case of the exceptional groups, these buildings are quite often a main effective tool,
and the algebraic description that goes with them, is invaluable in order to perform explicit calculations
involving exceptional groups.

The cases where the relative rank is at least two are relatively well understood,
mainly because of the work of Jacques Tits and Richard Weiss in the theory of Moufang polygons \cite{TW}.
However, for the Moufang quadrangles of exceptional type, mainly those of type $E_6$, $E_7$ and $E_8$,
the construction is rather ad-hoc, and a deeper explanation is still missing.

The goal of our paper is to give an explicit but at the same time completely intrinsic method to construct a family of rank two groups corresponding to the Moufang quadrangles of type $E_6$, $E_7$ and $E_8$ in \chacha\ different from two.

We will be dealing with the forms given by the following Tits indices:
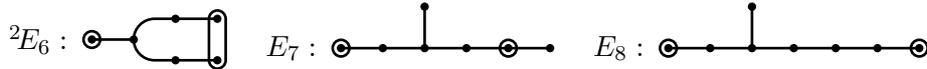
\[
\begin{tikzpicture}[line width=1pt, scale=.55]
	\node[left=5pt] (0,0) {$^2\!E_6:$};
	\draw (0,0) -- (1,0);
	\draw (1,0) arc (180:90:.5) -- (3,.5);
	\draw (1,0) arc (180:270:.5) -- (3,-.5);
	\diagnode{(0,0)}
	\diagnode{(1,0)}
	\diagnode{(2,.5)} \diagnode{(2,-.5)}
	\diagnode{(3,.5)} \diagnode{(3,-.5)}
	\distorbit{(0,0)}
	\distlongorbit{3}
\end{tikzpicture}
\quad
\begin{tikzpicture}[line width=1pt, scale=.55]
	\node[left=5pt] (0,0) {$E_7:$};
	\draw (0,0) -- (5,0);
	\draw (2,0) -- (2,1);
	\diagnode{(0,0)}
	\diagnode{(1,0)}
	\diagnode{(2,0)} \diagnode{(2,1)}
	\diagnode{(3,0)}
	\diagnode{(4,0)}
	\diagnode{(5,0)}
	\distorbit{(0,0)}
	\distorbit{(4,0)}
\end{tikzpicture}
\quad
\begin{tikzpicture}[line width=1pt, scale=.55]
	\node[left=5pt] (0,0) {$E_8:$};
	\draw (0,0) -- (6,0);
	\draw (2,0) -- (2,1);
	\diagnode{(0,0)}
	\diagnode{(1,0)}
	\diagnode{(2,0)} \diagnode{(2,1)}
	\diagnode{(3,0)}
	\diagnode{(4,0)}
	\diagnode{(5,0)}
	\diagnode{(6,0)}
	\distorbit{(0,0)}
	\distorbit{(6,0)}
\end{tikzpicture}
\]
Observe that all three forms have the property that their anisotropic kernel is of type $D_n$ (with an additional factor $A_1$ for the form of type $E_7$); this implies that these forms will be determined by an anisotropic quadratic form of dimension $2n$ having certain additional properties.
(The $A_1$ factor in the $E_7$ case gives rise to a quaternion division algebra, which turns up in the description of the Hasse invariant of the quadratic form corresponding to the $D_n$ factor.)
We will refer to such quadratic forms as forms of type $E_6$, $E_7$ and $E_8$, respectively.

In each case, we will see that the quadratic form can be characterized as the anisotropic part of the {\em Albert form} of a certain tensor product of composition algebras, and in fact, these algebras themselves will play a crucial role in the understanding of the corresponding algebraic groups;
they completely determine the algebraic group up to isogeny.

Our approach will turn out to be applicable in a more general situation, and in fact,
we will obtain {\em every possible} Moufang quadrangle defined over a field of characteristic different from two
starting from certain modules over a Jordan algebra.
Our construction relies in an essential way on the theory of $J$-ternary algebras and their Peirce decomposition, see \cite[Sections~3.12 and~6.61]{ABG}.

We can summarize our main result, namely the explicit construction of the quadrangular algebras of type $E_6$, $E_7$ and $E_8$, as follows.
\begin{construction*}Let $\cha(k)\neq2$. We start with a quadratic space $(k,V, q)$ of type $E_6$, $E_7$ or $E_8$ with base point
(see also Definition~\ref{def:e6e7e8} below).

By Theorem \ref{th:e6e7e8}, there exist an octonion division algebra $C_1$  and a division composition algebra $C_2$ of dimension $2$, $4$ or $8$, respectively such that $C_1$ and $C_2$ contain an isomorphic quadratic field extension, but no isomorphic quaternion algebra, and such that $q$ is similar to the anisotropic part of the Albert form, $q_A$, of $C_1\otimes_k C_2$. It follows that there exist $\ii_1\in C_1$ and $\ii_2\in C_2$ such that $\ii_1^2=\ii_2^2=a\in k\setminus\{k^2\}$. 

We define a subspace $V$ of the skew-elements of $C_1\otimes_k C_2$ of dimension $6$, $8$ or $12$, respectively, as\footnote{The orthogonal complement is taken w.r.t.\@ the bilinear form associated to the Albert form; this quadratic form is defined on the skew elements of $C_1\otimes_k C_2$.} 
\[V:=\langle  \ii_1\otimes1,1\otimes \ii_2\rangle^\perp.\] 
We choose an arbitrary $u\in V\setminus\{0\}$ and define the quadratic form
\[ Q:=\frac{1}{q_A(u)}q_A|_V ; \]
this form has base point $u$ and is similar to the quadratic form of type $E_6, E_7$ or $E_8$ we started with.

We then define the subspace $X_0$ of $C_1\otimes_k C_2$ of dimension $8$, $16$ or $32$ as 
\[X_0:= \Bigl\langle \bigl( x \otimes y + \tfrac{1}{a} \, \ii_1 x \otimes \ii_2 y \bigr) \mid x \in C_1, y \in C_2\Bigr\rangle .\]
Next, we define a suitable element $r\in \Ss$ as in Definition \ref{def:Je0e1u}\eqref{defr} below,
and we define the bilinear map $X_0\times L_0\rightarrow X_0$ as \[x\cdot v=v(r(u(r x))),\] 
and the bilinear map $h \colon X_0\times X_0\rightarrow V$ as \[h(x,y)=(u(r x))\overline{y}-y ((\overline{x}r)u).\]

In Theorem \ref{th:E6E7E8} we prove that the $7$-tuple $(k,V,Q,u,X_0,\cdot,h)$ is a quadrangular algebra of type $E_6$, $E_7$ or $E_8$, respectively.
It follows that this is the structure described in \cite[Chapter 13]{TW}
giving rise to the Moufang quadrangles of type $E_6$, $E_7$ and $E_8$, and hence to the corresponding
rank two forms of exceptional linear algebraic groups of type $E_6$, $E_7$ and $E_8$.
\end{construction*}

\paragraph{Organization of the paper}
In Section~\ref{sec2} we give some preliminar material on composition algebras, tensor products of composition algebras, quadrangular algebras and Peirce decomposition in Jordan algebras.

In Section~\ref{sec3} we work towards the main theorem of our paper, which is Theorem \ref{th:E6E7E8} and which gives a construction of quadrangular algebras of type $E_6$, $E_7$ and $E_8$. 

In Section~\ref{sec31} we show that we can construct  a quadrangular algebra in characteristic not 2 starting from a specific kind of module for a Jordan algebra of reduced spin type (see Definition \ref{def: redspin}).

In Section~\ref{sec32} we construct, in a similar way,  a quadrangular algebra starting from a $J$-ternary algebra over a Jordan algebra of reduced spin type. This section deals only with fields of \chacha\ not 2 nor 3, since $J$-ternary algebras are only defined over such fields.

In Section~\ref{sec33} we show that we can apply this procedure to obtain each quadrangular algebra of pseudo-quadratic form type in \chacha\ not $2$.
In Section~\ref{sec34} we construct (see the construction above) a module for a Jordan algebra of reduced spin type out of the tensor product of two composition algebras and show that this gives rise to quadrangular algebras of type $E_6,E_7$ and $E_8$ in \chacha\ not 2.

In Section~\ref{sec4}, inspired by Theorem \ref{constr quad alg}, we give a uniform description of all Moufang quadrangles in \chacha\ different from two.

\paragraph*{Acknowledgments}

Some of our ideas were inspired by fruitful discussions with Skip Garibaldi, in particular during a longer visit of the first author at Emory University,
whose hospitality is gratefully acknowledged.
Skip Garibaldi's observation mentioned in Theorem~\ref{th:e6e7e8} was a crucial first step in the whole project.
We also thank Bruce Allison for fruitful discussions about $J$-ternary algebras.
Last but not least, we are greatly indebted to the referee for a spectacularly detailed and careful reading of our paper.

\section{Preliminaries}\label{sec2}

We assume throughout the paper that $k$ is a commutative field of characteristic different from 2.
\subsection{Composition algebras}

A {\em composition algebra} is a, not necessarily commutative nor associative, unital $k$-algebra $C$ equipped with a quadratic form $q\colon C\to k$ that is multiplicative, i.e.\@ $q(xy)=q(x)q(y)$ for all $x,y\in C$. This quadratic form $q$ is called the {\em norm form}, its associated bilinear form will be denoted by $f$. With the norm form we associate an involution on $C$ by defining
\[\si\colon C\to C\colon x\mapsto\overline{x}:=f(x,1)1-x.\] 
By a classical result (see for example \cite[Theorem 1.6.2]{SV}) each composition algebra has dimension 1, 2, 4 or 8:
\begin{compactenum}[\rm (i)]
	    \item
		If $\dim_k C = 1$, then $C=k$, $q(x)=x^2$ and the involution is trivial.
	    \item
		If $\dim_k C = 2$, then $C/k$ is a {\em quadratic \'etale extension} of $k$. There exists $a\in k$ such that $C=k[\ii]/(\ii^2-a)$, the norm form is 				$\langle 1,-a\rangle$.
		Either $C/k$ is a separable quadratic field extension and $\sigma$ is the non-trivial
		element of $\Gal(C/k)$, or $C \cong k \oplus k$ and $\sigma$ interchanges the two components.
	    \item
		If $\dim_k C = 4$, then $C/k$ is a {\em quaternion algebra} over $k$. There exist $a,b\in k$ such that $C=k\oplus k\ii\oplus k \jj\oplus k(\ii\jj)$ with multiplication defined by \[\ii^2=a, \jj^2=b,\ii\jj=-\jj\ii.\] This quaternion algebra is denoted by $(a,b)_k$. The norm form is equal to $\langle1, -a\rangle \langle 1,-b\rangle$, the involution fixes $k$ and maps $\ii\mapsto -\ii$, $\jj\mapsto -\jj$.
	    \item
		If $\dim_k C = 8$, then $C/k$ is an {\em octonion algebra} over $k$. There exist $a,b,c\in k$ such that $C=Q\oplus Q\kk$ where $Q=(a,b)_k$ and multiplication is given by \[(x_1+x_2\kk)(y_1+y_2\kk)=(x_1y_1+c\overline{y_2} x_2)+(y_2x_1+{x_2}\overline{y_1} )\kk \quad\text{for all }x_i,y_i\in Q.\]
The norm form is $\langle1, -a\rangle \langle 1,-b\rangle \langle 1,-c\rangle$ and the involution is given by $\overline{x_1+x_2\kk}=\overline{x_1}-x_2\kk$. for all $x_1,x_2\in Q$.
	\end{compactenum}

In each case, the norm form is a Pfister form, these are forms of dimension $2^n$ denoted by $\llangle a_1,\dots,a_n\rrangle:=\otimes_{i=1}^n\langle 1, a_i\rangle$ for $a_1,\dots,a_n\in k$.
	The norm form is anisotropic when $C$ is a division algebra,
	and it is hyperbolic otherwise (i.e.\@ when $C$ is a split algebra).	

The norm form is completely determined by the algebra structure of the composition algebra.
It is a well known but somewhat deeper fact (see e.g.\@ \cite{SV}) that the converse also holds, i.e.
the composition algebra is determined up to isomorphism by the (similarity class of) the norm.
	
Quaternion algebras are not commutative, but associative. 
Octonion algebras are neither commutative nor associative. In the lemma below we summarize some useful identities that hold in each composition algebra.

\begin{lemma}[{\cite[Lemma 1.3.2, 1.3.3 and 1.4.1]{SV}}]\label{identities octonion}

Let $C$ ben an arbitrary composition algebra with norm $q$, with associated bilinear form $f$, and involution denoted by $x\mapsto \overline{x}$. Then for all $x,y,z\in C$ we have
\begin{compactenum}[\rm (i)]
 \item$ x^2-f(x,e)x+q(x)e=0, 
 $\item$ f( xy,z )= f( y, \overline{x}z), \quad   f( xy,z )= f( z, y\overline{x}), \quad   f( xy,z )= f( y\overline{z}, \overline{x})$,  
\item\label{alt} Each subalgebra generated by two elements is associative,
  \item$ x(\overline{x}y)=q(x)y, \quad (x\overline{y})y=q(y)x, 
 $\item\label{moufang id}$ (zx)(yz)=z((xy)z),\quad z(x(zy))=(z(xz))y,\quad x(z(yz))=((xz)y)z$.
\end{compactenum}
\end{lemma}
Property \eqref{alt} is called the alternativity, the identities in \eqref{moufang id} are called the {\em Moufang identities}.

\subsection{Tensor products of composition algebras}
We now assume that $C_1$ and $C_2$ are two composition algebras over $k$ (possibly of different dimension),
with norm forms $q_1$ and $ q_2$ and involutions $\sigma_1$ and $\sigma_2$, respectively.
Consider
\[ C_1 \otimes_k C_2 , \]
equipped with the involution
\[ \si:=\sigma_1 \otimes \sigma_2 . \]
If $\cha(k)\neq 2,3$ the algebra $(C_1 \otimes_k C_2, \sigma_1\otimes \sigma_2)$ is a {\em structurable algebra}. This is a class of algebras that generalizes Jordan algebras and associative algebras with involution. We will not need the exact definition and refer the interested reader to \cite{struct alg}.

Let $S_i$ be the set of skew elements in $C_i$, i.e.
\[ S_i = \{ x \in C_i \mid \overline{x}:=x^{\si_i} = -x \}, \]
and similarly, let $S$ be the set of skew elements of $C_1 \otimes_k C_2$, i.e.
\[ \Ss = \{ x \in C_1 \otimes_k C_2 \mid \overline{x}:=x^\si = -x \} = (S_1 \otimes 1) \oplus (1 \otimes S_2); \]
observe that $\dim_k \Ss = \dim_k C_1 + \dim_k C_2 - 2$.

\begin{definition}
We will associate a quadratic form $q_A$ to $C_1 \otimes_k C_2$, called the {\em Albert form}, by setting
\[ q_A \colon \Ss \to k \colon (x \otimes 1) + (1 \otimes y) \mapsto q_1(x) - q_2(y) \]
for all $x \in S_1$ and $y \in S_2$. When we denote $q'_i=q_i|_{S_i}$ for the pure part of the Pfister form $q_i$, we have $q_A=q'_1\perp \langle-1\rangle q'_2$.
\end{definition}

This form is named after A.A.~Albert, who studied the case where $C_1$ and $C_2$ are both quaternion algebras,
i.e.\@ $C_1 \otimes_k C_2$ is a biquaternion algebra.

\begin{definition}\label{def:invsharp}
Let $s=s_1\otimes 1+1\otimes s_2\in S$, we define the map {\em sharp} by
\[(s_1\otimes 1+1\otimes s_2)^\natural=s_1\otimes 1-1\otimes s_2.\]
If $q_A(s)\neq 0$, the {\em inverse} of $s$ is defined by
\[s^{-1}:=-\frac{1}{q_A(s)}s^\natural.\]
\end{definition}

Tensor products of two composition algebras are far from associative or alternative, but the skew elements behave nicer than arbitrary elements:

\begin{lemma}\label{ident bioctonion}
For all $x\in C_1\otimes_kC_2$, $s_1,s_2,s\in \Ss$ we have that
\begin{compactenum}[(\rm i)]
\item $s_1(s_2s_1)=(s_1s_2)s_1$
\item $(s_1s_2s_1)x=s_1(s_2(s_1x))$
\item If $s$ is invertible then $s(s^{-1}x)=x$.
\end{compactenum}
\end{lemma}
\begin{proof}
These identities can be easily checked using Lemma \ref{identities octonion}.
\end{proof}

\begin{remark}
In the case that  both $C_1$ and $C_2$ are quaternion algebras, $C_1\otimes_k C_2$ is associative and A.A.\@ Albert proved that $C_1\otimes_k C_2$ is a division algebra if and only if its Albert form is anisotropic (see \cite[Theorem  III.4.8]{L}.)

It is not obvious to generalize this result to arbitrary composition algebras. (Notice that, in the theory of structurable algebras, the concept of conjugate invertibility is used.) In \cite[Theorem 5.1]{A2} it is proven in the case that  $\cha(k)=0$ that the tensor product of two octonion algebras is a conjugate division algebra if and only if the corresponding Albert form is anisotropic.
To the best of our knowledge, it is an open problem whether this equivalence also holds for fields of characteristic $>3$.
\end{remark}

The case where $q_A$ has Witt index one will be needed to study the rank two forms of linear algebraic groups of type $E_6, E_7, E_8$ discussed in the introduction.

\begin{definition}[{\cite[Definition 5.11]{L}}]
Two $n$-fold Pfister forms $q_1$, $q_2$ are $r$-linked if there is an $r$-fold Pfister form $h$ such that $q_1\simeq h\otimes q_3$ and $q_2\simeq h\otimes q_4$ for some Pfister forms $q_3, q_4$.

The {\em linkage number} of $q_1$ and $q_2$ is the number $r\in \mathbb{N}$ such that $q_1$ and $q_2$ are $r$-linked but not $(r+1)$-linked.
\end{definition}

\begin{lemma}\label{lem:wittindex}
Let $C_1$ be an octonion division algebra with norm $q_1$ and let $C_2$ be a separable quadratic field extension, quaternion division algebra or an octonion division algebra, with norm $q_2$. The following are equivalent:
\begin{compactenum}[\rm (i)]
\item\label{item1}  $C_1$ and $C_2$ contain isomorphic separable quadratic field extensions, but $C_1$ and $C_2$ do not contain isomorphic quaternion algebras.
\item\label{item2} The linkage number of $q_1$ and $q_2$ is 1, i.e. $q_1$ and $q_2$ are 1-linked but not 2-linked.
\item\label{item3} The Witt index of the Albert form $q_A$ of $C_1\otimes_k C_2$ is equal to one.
\end{compactenum}
\end{lemma}
\begin{proof} 
Since the Witt index of $q_A$ is one less than the Witt index of $q_1\perp -q_2$, the equivalence of \eqref{item2} and \eqref{item3} is given by a result of Elman--Lam (see for example \cite[Theorem X.5.13]{L}).

The following observations follow from \cite[Prop. 1.5.1]{SV}.
Let $C$ be a composition algebra over $k$ with norm $q$.
\begin{compactdesc}\setlength{\itemsep}{1ex}
\item[\normalfont Let $\dim(C)=4$ or $8$.] Then $C$ contains a separable extension field isomorphic to $k(\ii)/(\ii^2-a)$ with $a\in k$  if and only if there exists a Pfister form $\varphi$, of dimension $2$ or $4$ respectively, such that $q\simeq \llangle - a\rrangle\otimes \varphi$.
\item[\normalfont Let $\dim(C)=8$.] Then $C$ contains a quaternion algebra isomorphic to $(a,b)_k$ with $a,b\in k$ if and only if $q\simeq \llangle  -a,-b,-c\rrangle$ for some $c\in k$.
\end{compactdesc}
From this it follows immediately that \eqref{item1} and \eqref{item2} are equivalent. 
\end{proof}

\begin{remark}
Suppose that $C_1$ and $C_2$ contain isomorphic separable quadratic field extensions. Even if they do not contain isomorphic quaternion algebras, it is still possible that $C_1$ and $C_2$ contain more than one isomorphic separable quadratic field extension up to isomorphism.
\end{remark}

\begin{definition}
We define the {\em linkage number} of $C_1$ and $C_2$ as the linkage number of their norm forms $q_1$ and $q_2$.
\end{definition}

Lemma~\ref{lem:wittindex} indicates that we will be particularly interested in pairs of composition algebras $C_1$, $C_2$ with linkage number one.

\subsection{Quadrangular algebras}

A quadrangular algebra is an algebraic structure that was constructed to describe the exceptional Moufang quadrangles. In Section~\ref{se:quadrangularsystems} we explain how one constructs Moufang quadrangles out of quadrangular algebras. For more information on quadrangular algebras, including characteristic 2, we refer to \cite{W}. We emphasize that the structure of a quadrangular algebra simplifies significantly in \chacha\ different from 2, see Remark \ref{charnot2}. Since this is the only case we will be dealing with we restrict our definition to $\cha(k)\neq2$.

\begin{definition}\label{def:quad}
A {\em quadrangular algebra}, in \chacha\  different from 2,  is an $7$-tuple $(k,L,q,1,X,\cdot,h)$, where
\begin{compactenum}[(\rm i)]
    \item
	$k$ is a commutative field with $\Char(k) \neq 2$,
    \item
	$L$ is a $k$-vector space,
    \item
	$q$ is an anisotropic quadratic form from $L$ to $k$,
    \item
	$1 \in L$ is a {\em base point} for $q$, i.e.\@ an element such that $q(1) = 1$,
    \item
	$X$ is a non-trivial $k$-vector space,
    \item
	$(x,v) \mapsto x \cdot v$ is a map from $X \times L$ to $X$ (usually denoted simply by juxtaposition),
    \item
	$h$ is a map from $X \times X$ to $L$, 
\end{compactenum}
satisfying the following axioms, where
\begin{align*}
	&f \colon L \times L \to k \colon (x,y) \mapsto f(x,y) := q(x+y) - q(x) - q(y) \,; \\
	&\sigma \colon L \to L \colon v \mapsto f(1,v) 1 - v \,; \\
	&v^{-1} := v^\sigma / q(v) \,.
\end{align*}
\begin{compactitem}
    \item[(A1)]
	The map $\cdot$ is $k$-bilinear.
    \item[(A2)]
	$x \cdot 1 = x$ for all $x \in X$.
    \item[(A3)]
	$(xv)v^{-1} = x$ for all $x \in X$ and all $v \in L^*$.
    \medskip
    \item[(B1)]
	$h$ is $k$-bilinear.
    \item[(B2)]
	$h(x,yv)=h(y,xv)+f(h(x,y),1)v$ for all $x,y \in X$ and all $v \in L$.
    \item[(B3)]
	$f(h(xv,y),1) = f(h(x,y),v)$ for all $x,y \in X$ and all $v \in L$.

    \medskip
    \item[(C)]
	$\theta(x,v) := \half h(x, xv)$.
    \medskip
    \item[(D1)]
	Let $\pi(x) = \theta(x,1)$ for all $x \in X$. Then $x \theta(x, v) = (x\pi(x))v$
	for all $x \in X$ and all $v \in L$.
    \item[(D2)]
	For all $x\in X\minuszero$ we have $\pi(x)\neq 0$.

\end{compactitem}
Moreover, we define a map $g \colon X \times X \to k$ by
\[ g(x,y) := \half f(h(x,y), 1) \]
for all $x,y \in X$.
\end{definition}

\begin{remark}\label{charnot2}
When one compares our definition of quadrangular algebras with the general definition in \cite[Definition 1.17]{W} there are two differences which are due to the fact that the definition simplifies when the \chacha\ is different from 2.
\begin{compactenum}[(\rm i)]
\item  The axiom (C) in \cite{W} consists of 4 more involved axioms (see \cite[Remark 4.8]{W}). By assuming $\theta(x,v)=\half h(x,xv)$ we actually assume that the quadrangular algebra is standard. Every quadrangular algebra is equivalent to a standard quadrangular algebra (see \cite[Proposition 4.2]{W}.)
\item In \cite{W}, axiom (D2) says $\pi(x) \equiv 0 \pmod{k}$ if and only if $x = 0$ (where $k$ has been identified
	with its image under the map $t \mapsto t\cdot 1$ from $k$ to $L$). We show that this is equivalent to our axiom (D2).

Assume the above (D2) holds. Applying (B2) with $x=y$, $v=1$ we get $f(h(x,x),1)=0$. If we suppose $\pi(x)=\half h(x,x)\in k1$, we have $f(h(x,x),1)=2h(x,x)=0$ and it follows that $\pi(x)=0$, so $x=0$.
\end{compactenum} 

\end{remark}
\begin{theorem}\label{th:class quad alg} A quadrangular algebra in characteristic not 2 is either obtained from an anisotropic pseudo-quadratic space over a quadratic pair
(see Section~\textup{\ref{def:pqs}}) or is of type $E_6,E_7$ or $ E_8$ (see Section~\textup{\ref{se:e6e7e8}}).
\end{theorem}
\begin{proof}
Since the \chacha\ of $k$ is not 2, it follows from \cite[2.3 and 2.4]{W} that the quadrangular algebra is regular, i.e. $f$ is non-degenerate (from \cite[3.14]{W} it follows that it is also proper, i.e. $\si\neq 1$). Now it follows from \cite[3.2]{W} that if the quadrangular algebra is not special (i.e. not arising from a pseudo-quadratic space) it is of type $E_6, E_7$ or $E_8$.
\end{proof}

\subsubsection{Pseudo-quadratic spaces}\label{def:pqs}

\begin{definition}[{\cite[Definition 1.16]{W}}]
A {\em pseudo-quadratic space} over a field of characteristic not 2 is a quintuple $(L,\si,X,h,\pi)$ where
\begin{compactenum}[(\rm i)]
    \item
	$L$ is a skew field of characteristic different from 2;
    \item
	$\si$ is an involution of $L$, and
	we let \[L_\si := \{\ell\in L\mid \ell^\si=\ell\}=\{\ell+\ell^\si\mid \ell \in L\};\]
    \item
	$X$ is a right vector space over $L$;
    \item
	$h \colon X\times X \rightarrow L$ is a {\em skew-hermitian form}, i.e.
	\begin{compactitem}
	\item $h$ is bi-additive and $h(x,yu)=h(x,y)u$,  and
	\item $h(x,y)^\si=-h(y,x)$,
	\end{compactitem}
	for all $x,y \in X$ and all $u\in L$;
    \item
	$\pi$ is a {\em pseudo-quadratic form} from $X$ to~$L$, i.e.
	\begin{compactitem}
	\item $\pi(x+y)\equiv \pi(x)+\pi(y)+h(x,y) \mod L_\si$, and
	\item $\pi(xu)\equiv u^\si\pi(x) u \mod L_\si$,
	\end{compactitem}
	for all $x,y \in X$ and all $u \in L$.
	Since we work in characteristic not 2 we can always assume that the pseudo-quadratic space is standard, i.e. $\pi(x)=\half h(x,x)$.
\end{compactenum}
A pseudo-quadratic space $(L,\si,X,h,\pi)$ is called {\em anisotropic} if
\[\pi(x)\equiv0 \mod L_\si \ \text{ only if }x=0.\]
\end{definition}

Not every pseudo-quadratic space is a quadrangular algebra; to be a pseudo-quadratic space the skew field has to satisfy some additional properties.

\begin{definition}[{\cite[Definition 1.12]{W}}]\label{def:quadratic pair}
Let $L$ be a skew-field with involution $\si$. We call $(L,\si)$ a {\em quadratic pair%
\footnote{This notion, taken from \cite[Definition 1.12]{W}, is quite different from the notion of a quadratic pair as defined in the Book of Involutions \cite{KMRT},
and has nothing to do with the notion of a quadratic pair in (finite) group theory either.}},
if $k:=L_\si$ is a field and if either
\begin{compactenum}[(\rm i)]
    \item
	$L/k$ is a separable quadratic field extension and $\si$ is the generator of the Galois group; or
    \item
	$L$ is a quaternion algebra over $k$ and $\si$ is the standard involution.
\end{compactenum}
Define $q(u)=uu^\si$; then $(k,L,q,1)$ is a pointed anisotropic non-degenerate quadratic space.
\end{definition}

 A result of Dieudonn\'e (see for example \cite[Theorem 1.15]{W}) says that if $\si$ is not trivial, the either $L$ is generated by $L_\si$ as a ring or $(L,\si)$ is a quadratic pair (in this case $L_\si$ is a field). From this point of view quadratic pairs are an exceptional class of skew-fields.

In \cite[Proposition 1.18]{W} it is shown that a non-zero standard anisotropic pseudo-quadratic space over a quadratic pair gives rise to the quadrangular algebra \[(k,L,q,1,X,\text{ scalar  multiplication},h).\]

\subsubsection{Quadrangular algebras of type $E_6,E_7$ and $E_8$}\label{se:e6e7e8}

For an explicit description of quadrangular algebras of type $E_6$, $E_7$ and $E_8$, we refer to \cite[Chapter 12 and 13]{TW};
for a concise description we refer to the first part of \cite[Chapter 10]{W}. 
Some care is needed, since the map $g$ in \cite{TW} is equal to $-g$ in \cite{W}.
Here we only give a concise overview of the structure of a quadrangular algebra of type $E_6$, $E_7$ or $E_8$.

\begin{definition}\label{def:e6e7e8}
A {\em quadratic space} $(k,L,q)$ with base point is of {\em type $E_6$, $E_7$ or $E_8$} if it is anisotropic and there exists a separable quadratic field extension $E/k$,
with norm denoted by $N$, such that:
\begin{compactitem}
\item[$E_6:$ ] there exist $s_2,s_3\in k^*$ such that \[(k,L,q)\cong (k,E^3,N\otimes\langle 1,s_2, s_3 \rangle);\]
\item[$E_7:$ ] there exist $s_2,s_3,s_4\in k^*$ such that $s_2s_3s_4\notin N(E)$ and \[(k,L,q)\cong (k,E^4,N\otimes \langle 1, s_2, s_3 , s_4\rangle);\]
\item[$E_8:$ ] there exist $s_2,s_3,s_4,s_5,s_6\in k^*$ such that $-s_2s_3s_4s_5s_6\in N(E)$ and \[(k,L,q)\cong (k,E^6,N\otimes \langle 1, s_2, s_3, s_4, s_5,s_6\rangle).\]
We always assume that $s_2s_3s_4s_5s_6=-1$, which can be achieved by rescaling the quadratic form if necessary.

\end{compactitem}
\end{definition}
As we are working in characteristic not $2$, we can choose $\ga\in E$ such that $E=k(\gamma)$ and $\gamma^2\in k$.
 
It is shown in \cite[(12.37)]{TW} that if
\[ (k,E^6,N\otimes \langle 1, s_2, s_3, s_4, s_5,s_6\rangle) \]
is a quadratic space of type $E_8$, then $(k,E^4,N\otimes \langle 1, s_2, s_3 , s_4\rangle)$ is a quadratic space of type $E_7$ and $(k,E^3,N\otimes\langle 1,s_2, s_3 \rangle)$ is a quadratic space of type $E_6$.

If $(k,L,q)$ is a quadratic space of type $E_6$, $E_7$ or $E_8$  with base point, there exists a scalar multiplication $E\times L\rightarrow L$ that extends the scalar multiplication $k\times L\rightarrow L$.

Let $(k,L,q,1,X,\cdot,h)$ be a quadrangular algebra of type $E_6$, $E_7$ or $E_8$, then $(k,L,q)$ is a quadratic space of type $E_6$, $E_7$ or $E_8$, respectively with basepoint denoted by $1$. This quadratic space determines the quadrangular algebra entirely (see \cite[Theorem 6.24]{W}).
 
The vector space $X$ has $k$-dimension $8$, $16$ or $32$, respectively; it is a $C(q,1)$-module (see Definition \ref{def:irmodule} below).
Some of the properties of the maps $\cdot, h, \theta$ and $\pi$ are given in Definition~\ref{def:quad}.
The existence of the vector space $X$ and of the maps $\cdot$, $h$ and $\theta$ is shown in \cite[Chapter 13]{TW}
by giving an explicit ad-hoc construction using the coordinatization of $L$. 

The goal of this article is to provide an alternative description of $X, L$ and the maps $\cdot, h$ starting from the tensor product of composition algebras.

In order to prove the anisotropy of our new construction of the map $\pi$ (see Theorem \ref{th:E6E7E8}), we need the concept of an irreducible $C(q,1)$-module.
\begin{definition}\label{def:irmodule}
\begin{compactenum}[(\rm i)]
\item Let $(k,V,q)$ be a quadratic space with basepoint 1, then the {\em Clifford algebra of $q$ with basepoint $1$} is
	defined as
	\[ C(q, 1) := T(V) / \langle u \otimes u^\sigma - q(u) \cdot 1 \rangle , \]
	where $T(V)$ is the tensor algebra of $V$, and where $\sigma$ is defined as in Definition~\ref{def:quad}.
	It is shown in \cite[12.51]{TW} that $C(q, 1) \cong C_0(q)$, the even Clifford algebra of $q$. The notion of a Clifford algebra with base point was introduced by Jacobson and McCrimmon; see \cite[Chapter 12]{TW} for more details.
	
\item\label{cq1module} Since $q$ is anisotropic, the axioms (A1)--(A3) of an arbitrary quadrangular algebra say precisely that $X$ is a $C(q, 1)$-module,  such that the action of $C(q,1)$ on $X$ is an extension of the action of $L$ on $X$ (see \cite[Proposition 2.22]{W}).
\item A $C(q, 1)$-module $X$ is {\em irreducible} if  we have $x\cdot C(q,1)=X$ for all $x\in X\minuszero$ .
\end{compactenum}

\end{definition}

\begin{theorem}[{\cite[2.26]{W}}]\label{th:irmodule}
Let $(k,V,q)$ be a quadratic space of type $E_6$, $E_7$ or $E_8$ with basepoint $1$ and let $X$ be a right $C(q,1)$-module. Then $X$ is an irreducible right $C(q,1)$-module if and only if $\dim_k(X)=8$, $16$ or $32$, respectively.
\end{theorem}

\subsection{ Peirce decomposition in Jordan algebras}

A good reference to study the theory of Jordan algebras is \cite{Mc}. Our construction of exceptional quadrangular algebras uses the Peirce decomposition of a Jordan algebra. We summarize the main properties and multiplication rules of Peirce subspaces.

\begin{definition}
A {\em Jordan $k$-algebra} $J$ is a unital commutative $k$-algebra such that for all $x,y\in J$ we have $(x^2y)x=x^2(yx)$.

We define the {\em $U$-operator} and its linearization for $x,y,z\in J$
\[U_xy:=2x(xy)-x^2 y, \quad U_{x,z}y:=(U_{x+z}-U_x-U_z)y.\]

An element $x\in J$ is {\em invertible} if and only if there exists a $y\in J$ such that $xy=1$ and $x^2y=x$;  this condition is equivalent with $U_xy=x, U_x y^2=1$. The element $y$ is the {\em inverse} of $x$.
\end{definition}
\begin{definition}[{\cite[II.8.1 and II.8.2 on p.235]{Mc}}]
Let $J$ be a Jordan $k$-algebra.
\begin{compactenum}[\rm (i)]
\item An element $e\in J$ is an {\em idempotent} if $e^2=e$. An idempotent is {\em proper} if it is different from 0 and 1.
If $e$ is an idempotent, then $1-e$ is also an idempotent. Two idempotents $e,e'$ are {\em supplementary} if $e+e'=1$. Observe that two supplementary idempotents are always orthogonal, i.e. $ee'=0$.

\item Let $e$ be a proper idempotent in $J$. The {\em Peirce decomposition} with respect to $e$ is defined as follows.
For each $i \in \Z\!\left\lbrack\half \right\rbrack$, let
\[J_i = \{x\in J \mid ex=ix\} ;\]
then we have
\[ J=J_0\oplus J_\subhalf \oplus J_1 ; \]
in particular, $J_\ell = 0$ if $\ell \not\in \{ 0, \half , 1 \}$. For a nondegenerate Jordan algebra we have $J_0\neq 0$ (see \cite[II.10.1.2]{Mc}).
Let $i\in \{0,1\}$ and $j=1-i$; then
\[J_i^2\subseteq J_i, \quad J_iJ_\subhalf\subseteq J_\subhalf, \quad J_\subhalf^2\subseteq J_0+J_1, \quad J_iJ_j=0 . \]
For all $\ell ,m\in \{0,\half ,1\}$, we have
\begin{align}\label{Upeirce}
U_{J_m} J_\ell \subseteq J_{2m-\ell}
\end{align}
This implies that $U_{J_m}J_\ell=0$ if $2m-\ell\notin\{0,\half,1\}$.
\end{compactenum}
\end{definition}

To construct quadrangular algebras we will use two types of Jordan algebras that contain supplementary idempotents. These are the Jordan algebras of reduced spin type and the Jordan algebras $\mathcal{H}(M_2(L),\si t)$ for a skew field $L$ with involution $\si$, where $t$ is the transpose map.

\begin{definition}[{\cite[II.3.4 on p.\@ 180]{Mc}}]\label{def: redspin}
Consider a quadratic form $q \colon V \to k$ over $k$. Starting from the vector space $V$ we construct a Jordan algebra by adjoining two supplementary idempotents to $V$.

As a vector space, we define $J$ by adjoining two copies of $k$ to $V$:  $J:=ke_0\oplus V\oplus ke_1$. We define the following multiplication:
\begin{align}
	(t_1e_i)(t_2 e_j) &= \delta_{ij}t_1t_2 e_i, \label{multeiej}\\
	(t e_i)v &= \half t v, \label{multeiv}\\
	v w &= \half f(v,w)(e_0+e_1),\label{multvw}
\end{align}
for all $i,j \in \{0,1\}, v,w \in V, t,t_1,t_2\in k$. This defines a Jordan algebra\footnote{Notice that this is the same Jordan algebra as the Jordan algebra of the quadratic form $Q:ke_0\oplus V\oplus ke_1\rightarrow k:t_0e_0+v+t_1e_1\mapsto t_0t_1-q(v)$.} on  $ke_0\oplus V\oplus ke_1$, called the {\em reduced spin factor of the quadratic form} $q$. We say that $J$ is of {\em reduced spin type}.

The unit of this Jordan algebra is $e_0+e_1$, and for all $v,w \in V$ we have
\[U_{v}e_0=q(v)e_1, \quad U_{v}e_1=q(v)e_0, \quad U_vw=f(v,w)v-q(v)w.\]

It is clear that $e_0$ and $e_1$ are supplementary idempotents and that we have the following Peirce subspaces with respect to $e_1$: 
\[J_0=ke_0, \quad J_\subhalf=V, \quad J_1=ke_1.\]
\end{definition}

\begin{definition}[{\cite[Example  II.3.2.4]{Mc}}]\label{def: matrix}
Let $L$ be a skew field with involution $\si$, define $L_\si:=\Fix_\si(L)$ and $k:=Z(L)$. The matrix algebra $M_2(L)$ is associative with involution $\si t$. Now let $J$ be the Jordan $k$-algebra\footnote{This Jordan algebra consists of the fixed points in $M_2(L)$ of the involution $\si t$; multiplication is given by $m.n=\half(mn+nm)$ (where on the right hand side the usual matrix multiplication is used).} \[\mathcal{H}(M_2(L),\si t)=\left\{\begin{bmatrix}\alpha_1&\ell^\si\\ \ell &\al_2\end{bmatrix}\mid \alpha_1,\al_2 \in L_\si,\ell\in L\right\}.\]

We define the supplementary idempotents $e_0=\left[\begin{smallmatrix}1&0\\0&0\end{smallmatrix}\right], e_1=\left[\begin{smallmatrix}0&0\\0&1\end{smallmatrix}\right]\in J$.
With respect to $e_1$, we have \[J_0=L_\si e_0,\quad J_1=L_\si e_1\quad \text{and}\quad J_\subhalf=\left\{\begin{bmatrix}0&\ell^\si\\ \ell &0\end{bmatrix}\mid \ell\in L\right\}.\] We have
\begin{align*}
	(\al_1e_i)(\al_2 e_j) &= \delta_{ij}\half(\al_1\al_2+\al_2\al_1) e_i, \\
	(\al e_0)v &= \half \begin{bmatrix}0&\al\ell^\si\\ \ell\al &0\end{bmatrix}, \\
	(\al e_1)v &= \half \begin{bmatrix}0&\ell^\si \al\\ \al\ell &0\end{bmatrix}, \\
	v_1 v_2&=\half(\ell_1^\si \ell_2+\ell_2^\si \ell_1) e_0+\half(\ell_1\ell_2^\si+\ell_2 \ell_1^\si)e_1.
\end{align*}
for all $i,j \in \{0,1\}, v=\begin{bmatrix}0&\ell^\si\\ \ell&0\end{bmatrix}, v_1=\begin{bmatrix}0&\ell_1^\si\\ \ell_1 &0\end{bmatrix}, v_2=\begin{bmatrix}0&\ell_2^\si \\ \ell_2 &0\end{bmatrix}\in J_\subhalf$, $\al, \alpha_1,\alpha_2\in L_\si $. 
For the $U$-operators we find
\[U_{v}(\al e_0)=(\ell\al \ell^\si  )e_1, \quad U_{v}(\al e_1)=(\ell^\si \al\ell) e_0, \quad U_{v_1}v_2=\begin{bmatrix}0&\ell_1^\si \ell_2\ell_1^\si \\ \ell_1\ell_2^\si \ell_1 &0\end{bmatrix}.\]
\end{definition}

\begin{remark}\label{mat quadr pair}
If we consider the above definition in the case that $(L,\si)$ is a quadratic pair (see Definition \ref{def:quadratic pair}), then $k=L_\si$. Now there exists a non-degenerate anisotropic quadratic form $q:L\rightarrow k:\ell\mapsto \ell\ell^\si=\ell^\si\ell$ and the Peirce subspaces $J_0,J_1$ of $\mathcal{H}(M_2(L),\si t)$ are one-dimensional.

Define a quadratic form $Q$ on $J_\subhalf\subseteq \mathcal{H}(M_2(L),\si t)$ given by $Q(\begin{bmatrix}0&\ell^\si \\ \ell &0\end{bmatrix})=q(\ell)$. By comparing the multiplication in Definition \ref{def: matrix} above and the one in Definition \ref{def: redspin} we conclude that $\mathcal{H}(M_2(L),\si t)$ is the reduced spin factor of the quadratic space $(J_\subhalf,k,Q)$.
\end{remark}

In the following Proposition we use Osborn's Capacity Two theorem to show that the two families of Jordan algebras we discussed above can be characterized in a unified way. The proof of this Proposition uses some results and concepts of Jordan theory that we will not use in the remaining part of this article.
\begin{proposition}\label{char red spin}
Let $J$ be a non-degenerate Jordan $k$-algebra with supplementary proper idempotents $e_0$ and $e_1$. Let $J_0,J_\subhalf,J_1$ be the Peirce subspaces of $J$ with respect to $e_1$. We assume that each element in $J_\subhalf\setminus\{0\}$ is invertible and that there exists $u\in J_\subhalf$ such that $u^2=1$.
\begin{compactitem}
\item If $\dim(J_0)=1$, then $J$ is the reduced spin factor of some non-degenerate anisotropic quadratic space with basepoint $u$.\footnote{Notice that $\mathcal{H}(M_2(L),\si t)$ for $(L,\si)$ a quadratic pair is included in this case.}
\item If $\dim(J_0)>1$, then $J$ is isomorphic to $\mathcal{H}(M_2(L),\si t)$ for some skew field $L$ with involution $\si$, such that $(L,\si)$ is not a quadratic pair.
\end{compactitem}
\end{proposition}
\begin{proof}
We will show that the assumptions imply that $J$ is a simple nondegenerate Jordan algebra of capacity 2; an algebra has capacity 2 if the unit is the sum of two supplementary idempotents $e_0,e_1$ such that the Peirce subspaces $J_0, J_1$ are division algebras. 

 Since $u^2=1$ it follows from \cite[II.6.1.10]{Mc} that $U_u$ is a Jordan isomorphism\footnote{This means that $U_u(xy)=U_u(x)U_u(y)$ for all $x,y\in J$.} of $J$ such that $(U_u)^2$ is the identity map. Since $U_u(J_1)\subseteq J_0$, $U_u$ is an isomorphism between $J_0$ and $J_1$. Therefore it is enough to show that $J_0$ is a division algebra. It follows from \cite[II.6.1.2]{Mc} that it is sufficient to show that for each element $t\in J_0\minuszero$ the operator $U_t$ is surjective on $J_0$.

Let $t,s\in J_0\minuszero$; using \cite[II.8.4.1]{Mc} one can verify that $U_{t}s=U_uU_{2ut}s$. We have $2ut\in J_\subhalf\minuszero$: if $2ut=0$ it would follow that $U_u(t)=-u^2t=-t$ which implies that $t\in J_0\cap J_1=\{0\}$. It follows that $U_{2ut}$ is invertible. Now let $r\in J_0$; since $U_{2ut}^{-1}U_u r\in J_0$ we have\[r=U_{t}(U_{2ut}^{-1}U_u r)\] and hence $U_{t}$ is surjective on $J_0$.

This proves that $J$ has capacity 2.
 
We now show that $J$ is simple. From \cite[II.20.2.4]{Mc} a nondegenerate algebra with capacity is simple iff its capacity is connected (i.e. if $e_0, e_1$ are connected \cite[II.10.1.3]{Mc}). In fact $e_0,e_1$ are even strongly connected since $u\in J_\subhalf$ is an involution, i.e. $u^2=1$. 

We proved all the conditions of Osborn's Capacity Two theorem, see \cite[II.22.2.1 on p. 351]{Mc}. This theorem states that a simple nondegenerate Jordan algebra of capacity 2 belongs to exactly one of the following three disjoint classes from which we can exclude the first:
\begin{compactenum}[(\rm i)]
\item Full type $M_2(L)^+$, for a noncommutative skew field $L$. The only idempotents are $e_0=\left[\begin{smallmatrix}1&0\\0&0\end{smallmatrix}\right]$, $e_1=\left[\begin{smallmatrix}0&0\\0&1\end{smallmatrix}\right]$. It is clear that the element $\begin{bmatrix}0&0\\\ell&0\end{bmatrix}\in J_\subhalf$ is not invertible.
\item Hermitian type $\mathcal{H}(M_2(L),\si t)$  with $L$ a skew-field with involution $\si$ such that $(L,\si)$ is not a quadratic pair. In this case $\dim(J_0)> 1$.
\item Reduced spin factor of a non-degenerate quadratic space $(k,V,q)$. Since the unit of $J$ is $e_0+e_1$ and $1=u^2=q(u)(e_0+e_1)=q(u)1$, $u$ is a base point of $q$. 

Suppose there exists a $v\in J_\subhalf\minuszero$ such that $q(v)=0$; then it would follow that $vw=0$ for all $w\in J_\subhalf$, which implies that $v$ is not invertible. Therefore $q$ is anisotropic. In this case $\dim(J_0)=1$.\qedhere
\end{compactenum}
\end{proof}

\section{A coordinate-free construction of quadrangular algebras}\label{sec3}

In this section we give a coordinate-free construction of quadrangular algebras.  Our construction was inspired by several properties of $J$-ternary algebras; see Definition \ref{def:Jtern} and \cite[3.12]{ABG}.

The entire article \cite{ABG} deals with fields of characteristic zero only. However the concept of a J-ternary algebra and its basic properties, such as Peirce decomposition (see Lemma \ref{peirce Jtern} and \cite[6.61]{ABG}) can be generalized without any adjustments to fields of characteristic different from 2 and 3.

It is not clear at all how to generalize the theory of $J$-ternary algebras to fields of characteristic 2 and 3; one reason is that the definition of a $J$-ternary algebra uses bilinear and trilinear forms. 

However in the theory of quadrangular algebras there is no difference between fields of characteristic different form $2$ and $3$ and fields of characteristic $3$. Therefore we want our construction of quadrangular algebras to work in \chacha\ 3 in the same way as in \chacha\ not 2 and 3. 

Actually, $J$-ternary algebras contain some axioms that are superfluous for our construction. In Theorem \ref{constr quad alg} we show that we can prove the axioms for quadrangular algebras using only a few axioms concerning a module for a Jordan algebra that has a skew-symmetric form. We replaced all the identities in the definition of a $J$-ternary system involving the trilinear form by an identity that only uses the bilinear form.  Therefore we do not need the assumption that the \chacha\ is different from 3. In Section~\ref{sec:Jtern} we show that a $J$-ternary algebra still satisfies the identity we demand in Theorem \ref{constr quad alg}.

 In the following subsection we do not start by giving a definition of $J$-ternary algebras. Instead we start by considering the concepts that we will need to formulate Theorem \ref{constr quad alg}.

To include \chacha\  2 would be another cup of tea for several reasons: already the definition of quadrangular algebras is much more complicated and the definition of a Jordan algebra is more delicate.

\subsection{Quadrangular algebras from special Jordan modules}\label{sec31}
Let $k$ be a field of \chacha\ different from 2. In the next lemma we introduce a module for Jordan algebras.
\begin{lemma}\label{lem: module}
Let $J$ be a Jordan $k$-algebra and let $X$ be a $k$-vector space. Suppose that $J$ acts on $X$ by $\bt:J\times X\rightarrow X$ such that 
\begin{compactenum}[\rm (i)]
\item $(tj)\bt x=j\bt (tx)$,
\item $(j+j')\bt x=j\bt x+ j'\bt x$,
\item $j\bt(x+y)=j\bt x+ j\bt y$,
\item $1\bt x=x$,
\end{compactenum}
for all $j,j'\in J,x,y\in X, t\in k$. Then the following are equivalent 
\begin{compactitem}
\item[\rm (v)]
$U_j j'\bt x=j\bt(j'\bt (j\bt x))$, 
\item[\rm (v')] $(jj')\bt x=\half (j\bt(j'\bt x)+j'\bt(j\bt x))$.
\end{compactitem}
\end{lemma}
\begin{proof}
Assume that (v) holds. Since $U_j 1=j^2$, we have $j^2\bt x=j\bt(j\bt x)$. Linearizing this expression gives us (v').

Assume that (v') holds, we have
\begin{align*}
U_j j'\bt x&=(2j(jj')-j^2j')\bt x\\
&=\half (2 j\bt((jj')\bt x)+2 (jj')\bt (j \bt x)\\
&\hspace*{30ex}-j^2\bt( j'\bt x)- j'\bt (j^2\bt x))\\
&=\half (j^2\bt( j'\bt x)+j'\bt(j^2\bt x)+2j\bt(j'\bt(j\bt x))\\
&\hspace*{30ex}-j^2\bt( j'\bt x)- j'\bt( j^2\bt x))\\
&=j\bt(j'\bt(j\bt x)). \tag*{\qedhere}
\end{align*}
\end{proof}
\begin{definition}[{\cite[3.12]{ABG}}]
Let $J$ be a Jordan $k$-algebra and let $X$ be a $k$-vector space with action $\bt:J\times X\rightarrow X$. Then $X$ is a {\em special $J$-module} if the  conditions (i)-(v)  of the previous lemma are satisfied.
\end{definition}

\begin{lemma}\label{lem: skewsym}
Let $X$ be a special $J$-module and let $\lsk \cdot , \cdot \rsk:X\times X\rightarrow J$ be a bilinear skew-symmetric form. Then 
\begin{align}\label{skew id} 
U_j\lsk x,y\rsk=\lsk j\bt x,j\bt y\rsk\iff j\lsk x, y\rsk=\half (\lsk j\bt x, y\rsk+\lsk x,j\bt y\rsk).
\end{align}

\end{lemma}
\begin{proof}
\begin{compactitem}
\item[$\Rightarrow:$] This follows from $U_{j+1}j'-U_j j'-U_1 j'=2jj'$.

\item[$\Leftarrow:$] We have
\begin{multline*}
\begin{aligned}
U_j\lsk x,y\rsk&=2j(j \lsk x,y\rsk)-j^2 \lsk x,y\rsk\\
&=\half \bigl( \lsk j\bt( j\bt x),y\rsk+ \lsk x, j\bt( j\bt y)\rsk+2  \lsk  j\bt x,j\bt y\rsk\bigr)
\end{aligned}\\
-\half \bigl(\lsk j^2 \bt x,y\rsk+\lsk  x,j^2 \bt y\rsk\bigr)
\end{multline*}
The result follows since $j\bt(j\bt x)=j^2\bt x$.\qedhere
\end{compactitem}
\end{proof}

In the following Lemma we consider the Peirce decomposition of special $J$-modules; see also \cite[6.61]{ABG}.

\begin{lemma}\label{peirce Jtern}
Let $J$ be a Jordan $k$-algebra with supplementary proper idempotents $e_0$ and $e_1$. Let $J_0,J_\subhalf,J_1$ be the Peirce subspaces of $J$ with respect to $e_1$. 
 Let $X$ be a special $J$-module and define 
\begin{align*}
X_0&:=\{x\in X\mid e_0 \bt x=x\}=\{x\in X\mid e_1 \bt x=0\},\\
 X_1&:=\{x\in X\mid e_0 \bt x=0\}=\{x\in X\mid e_1\bt x=x\}.
 \end{align*}
 Then
 \begin{compactenum}[\rm (i)]
 \item  We have $e_0\bt X=X_0$, $e_1\bt X=X_1$ and  $X=X_0\oplus X_1$.
\item For $i\in \{0,1\}$ and $j=1-i$ we have 
\begin{equation}\label{eq:peircemodule}J_i\bt X_i\subseteq X_i,\quad J_i\bt X_j=0,\quad J_\subhalf\bt X_i\subseteq X_j.\end{equation}

\item Let $\lsk \cdot , \cdot \rsk:X\times X\rightarrow J$ be a skew bilinear form satisfying \eqref{skew id}, then 
\begin{equation}\label{eq:skewpeirce}\lsk X_i,X_i\rsk\subseteq J_i,\quad \lsk X_i,X_j\rsk \subseteq J_\subhalf \end{equation}
for $i\in \{0,1\}$ and $j=1-i$.
\item Assume there exists an element $u\in J_\subhalf$ such that $u^2=1$. The map \[X_0\to X_1:x\mapsto u\bt x\] is a vector space isomorphism, called the {\em connecting morphism.}
\end{compactenum}
\end{lemma}

\begin{proof}
 \begin{compactenum}[\rm (i)]
 \item Let $i\in \{0,1\}$, we have $e_i\bt(e_i\bt x)=(e_i e_i)\bt x=e_i\bt x$ for all $x\in X$, therefore $e_i\bt X= X_i$. Since $(e_0+e_1)\bt x=x$, we have $X=X_0\oplus X_1$. 
 \item This follows by evaluating \[e_1\bt (j \bt x)=2(e_1j)\bt x-j\bt(e_1\bt x)\] for all combinations of $j\in J_0$, $J_1$ or $J_\subhalf$ and $x\in X_0$ or $X_1$, using the fact that $X=X_0\oplus X_1$.
 \item This follows from evaluating $e_1\lsk x,y\rsk=\half(\lsk e_1\bt x,y\rsk+\lsk x, e_1\bt y\rsk)$ for $x,y\in X_0$ or $X_1$. 
 \item It follows from (ii) that $u\bt x\in X_1$ iff $x\in X_0$. Since $u\bt (u\bt x)=(u u)\bt x=x$ the connecting morphism is an isomorphism.
\qedhere
 \end{compactenum}
\end{proof}

\begin{theorem} \label{constr quad alg}Let $\Char(k)\neq 2$.

Let $J$ be the reduced spin factor of the non-degenerate anisotropic quadratic space $(k,V,q)$ with base point $u$: $J=ke_0\oplus V\oplus ke_1$.

Let $X$ be a non-trivial special $J$-module equipped with a bilinear skew-symmetric form $\lsk \cdot , \cdot \rsk:X\times X\rightarrow J$ satisfying \eqref{skew id}.

Assume that the following holds:
\begin{align}
&\forall x\in X_0,v\in V:&&\lsk v\bt x,x\rsk \bt x=v\bt( u\bt(\lsk u\bt x,x\rsk \bt x)),\label{1evoorwaarde}\\
&\forall x\in X_0\minuszero:&& \lsk u\bt x,x\rsk\neq0.\label{2evoorwaarde}
\end{align}
We define 
\begin{align*}
\cdot&:X_0\times V\rightarrow X_0: x\cdot v=v\bt(u\bt x)\\
h&:X_0\times X_0\rightarrow V:(x,y)\mapsto\lsk u\bt x,y\rsk.
\end{align*}

Then $(k,V,q,u,X_0,\cdot,h)$ is a quadrangular algebra.
\end{theorem}
\begin{proof}
Notice that $e_0,e_1\in J$ are supplementary proper idempotents and that $u\in J_\subhalf$ such that $u^2=q(u)1=1$. Thus we can apply Lemma \ref{peirce Jtern} with $J_0=ke_0, J_\subhalf=V,J_1=ke_1$. It follows from \eqref{eq:peircemodule} and \eqref{eq:skewpeirce} that the maps $\cdot$ and $h$ are well defined. To start we show that $U_u(v)=v^\si$ for all $v\in J_\subhalf$ with $\si$ as in Definition \ref{def:quad}:
\begin{align*}
U_u(v)&=2u(uv)-v\\
&=u(f(u,v)1)-v\\
&=f(u,v)u-v=v^\si.
\end{align*}

We verify that all the axioms of a quadrangular algebra given in Definition \ref{def:quad} hold.
\begin{compactitem}
\item[(A1)] This follows from linearity of $\bt$.
\item[(A2)] Let $x\in X_0$, then $x\cdot u=u\bt(u\bt x)=u^2\bt x=1\bt x=x$.
\item[(A3)] Let $x\in X_0$ and $v\in J_\subhalf$, then
\begin{align*} 
(x\cdot v)\cdot v^\si&=U_u(v)\bt(u\bt (v\bt (u\bt x)))\\
&=u\bt(v\bt(v\bt(u\bt x)))\\
&=\half f(v,v) \ u\bt (1\bt (u\bt x))\\
&=q(v) x.
\end{align*}
\item[(B1)] This follows by from bilinearity of $\lsk .,.\rsk$.
\item[(B2)]  Let $x,y\in X_0$ and $v\in J_\subhalf$, then by applying consecutively \eqref{multvw}; \eqref{skew id} and (A2); Lemma \ref{lem: module}.(v) we find 
\begin{align*}
&&h(x,y\cdot v)&=h(y,x\cdot v)+f(h(x,y),u) v\\
&\iff&\lsk u\bt x,v\bt(u\bt y)\rsk&=\lsk u\bt y,v\bt(u\bt x)\rsk+f(\lsk u\bt x,y\rsk,u)v\\
&\iff & \lsk u\bt x,v\bt(u\bt y)\rsk\hspace*{5ex}\\
&&+\lsk v\bt(u\bt x),u\bt y\rsk&=2(\lsk u\bt x,y\rsk u)v\\
&\iff &2v\lsk u\bt x,u\bt y\rsk&=(\lsk u\bt x,u\bt y\rsk +\lsk x,y\rsk)v\\
&\iff & v U_u\lsk x,y\rsk&=\lsk x,y\rsk v
\end{align*}
From \eqref{eq:skewpeirce} we know that $\lsk x,y\rsk=t e_0$ for some $t\in k$. It follows from Definition \ref{def: redspin} that \[v U_u(t e_0)=v( t e_1)=\half t v=(t e_0)v.\]
\item[(B3)] Let $x,y\in X_0$ and $v\in J_\subhalf$, then by applying consecutively \eqref{multvw}; \eqref{skew id} we find 
\begin{align*}
&&f(h(x\cdot v,y),u)&=f(h(x,y),v)\\
&\iff&f(\lsk u\bt(v\bt(u\bt x)),y\rsk,u)&=f(\lsk u\bt x,y\rsk,v)\\
&\iff&\lsk u\bt(v\bt(u\bt x)),y\rsk u&=\lsk u\bt x,y\rsk v\\
&\iff&\lsk u\bt(v\bt(u\bt x)),u\bt y\rsk\hspace*{5ex}\\
&&+\lsk v\bt(u\bt x),y\rsk&=\lsk u\bt x,v\bt y\rsk+\lsk v\bt(u\bt x),y\rsk\\
&\iff&\lsk u\bt(v\bt(u\bt x)),u\bt y\rsk&=\lsk u\bt x,v\bt y\rsk\\
\intertext{Since $U_u=\si$ on $J_\subhalf$ is an involution, by \eqref{skew id} this is equivalent to}
&\iff&\lsk v\bt(u\bt x), y\rsk &=\lsk x,u\bt(v\bt y)\rsk\\
&\iff&\lsk v\bt(u\bt x),v\bt(v\bt y)\rsk&=q(v)\lsk u\bt(u\bt x),u\bt(v\bt y)\rsk
\end{align*}
the last equivalence follows from \eqref{multvw}. From \eqref{eq:skewpeirce} we know that $\lsk u\bt x,v\bt y\rsk=t e_1$ for some $t\in k$. The last equation reduces to 
\[U_v(t e_1)=q(v)U_u(t e_1)\]
This holds since $U_v(e_1)=q(v)e_0$ and $U_u(e_1)=e_0$ by  Definition \ref{def: redspin}.
\item[(C)] $\theta(x,v):=\half \lsk u\bt x,v\bt (u \bt x)\rsk$.
\item[(D1)] Let $x\in X_0$ and $v\in J_\subhalf$. Since $f$ is non-degenerate, we have $V=ku\oplus u^\perp$. Now
\begin{align*}
&&x\cdot h(x,x\cdot v)&= (x\cdot h(x,x))\cdot v\\
&\iff& \lsk u\bt x,v\bt (u\bt x)\rsk \bt (u\bt x)&=v\bt( u \bt (\lsk u\bt x,x\rsk \bt (u\bt x)))
\end{align*}
Since this expression is linear in $v$ and trivial for $v\in ku$, we can assume $v\in u^\perp$ and thus $f(u,v)=uv=0$ and hence $u\bt(v\bt x)=-v\bt(u\bt x)$. In this case we continue as follows:
\begin{align*}
&\iff& -\lsk u\bt x,u\bt (v\bt x)\rsk \bt (u\bt x)&=v\bt(U_u\lsk u\bt x, x\rsk \bt x)\\
&\iff& U_u\lsk x,v\bt x\rsk \bt (u\bt x)&=-v\bt(\lsk x,u\bt x\rsk \bt x)\\
&\iff& u\bt (\lsk x,v\bt x\rsk \bt x)&=-v\bt(\lsk x,u\bt x\rsk \bt x)\\
&\iff& \lsk x,v\bt x\rsk \bt x&=-u\bt(v\bt(\lsk x,u\bt x\rsk \bt x))\\
&\iff& \lsk x,v\bt x\rsk \bt x&=v\bt(u\bt(\lsk x,u\bt x\rsk \bt x)).
\end{align*}

This is exactly \eqref{1evoorwaarde}.
\item[(D2)] 
This is assumption \eqref{2evoorwaarde}.\qedhere
\end{compactitem}
\end{proof}

\subsection{Quadrangular algebras from $J$-ternary algebras}\label{sec:Jtern}\label{sec32}
In this subsection we assume $\cha(k)\neq 2,3$ and we prove that an arbitrary `anisotropic' non-trivial $J$-ternary algebra, where $J$ is as in Theorem~\ref{constr quad alg}, satisfies the assumptions of Theorem~\ref{constr quad alg}. It follows that we can construct quadrangular algebras from $J$-ternary algebras.

We remind the reader that the entire article \cite{ABG} is only written for fields of characteristic zero, but that the concept of a J-ternary algebra and its basic properties can be generalized without any adjustments to hold in fields of characteristic different from 2 and 3.

\begin{definition}\label{def:Jtern}
Let $\cha(k)\neq 2,3$, let $J$ be a Jordan $k$-algebra, let $X$ be a special $J$-module with action $\bt$.

Assume $\lsk\ ,\ \rsk:X\times X\rightarrow J$ is a skew-symmetric bilinear map, and $\lsk\ ,\  ,\ \rsk:X\times X\times X\rightarrow X$ is a trilinear product.
Then $X$ is called a $J$-ternary algebra if the following axioms hold for all $j\in J$, $x,y,z,v,w\in X$:

\begin{compactitem}\setlength{\itemsep}{.6ex}
\item[(JT1)] $j\lsk x, y\rsk=\half \lsk j\bt x,y\rsk+\half \lsk x, j\bt y\rsk \quad$ (This is the right side of \eqref{skew id}.) 
\item[(JT2)] $j\bt \lsk x,y,z\rsk=\lsk j\bt x,y,z\rsk -\lsk x,j\bt y,z\rsk +\lsk x,y,j \bt z\rsk$
\item[(JT3)] $\lsk x,y,z\rsk=\lsk z,y,x\rsk-\lsk x,z\rsk\bt y$
\item[(JT4)] $\lsk x,y,z\rsk=\lsk y,x,z\rsk+\lsk x,y\rsk\bt z$
\item[(JT5)] $\lsk \lsk x,y,z\rsk,w\rsk+\lsk z, \lsk x,y,w\rsk\rsk=\lsk x,\lsk z,w\rsk \bt y\rsk$
\item[(JT6)] $\lsk x,y,\lsk z,w,v\rsk\rsk=\lsk\lsk x,y,z\rsk,w,v\rsk+\lsk z,\lsk y,x,w\rsk,v\rsk+\lsk z,w,\lsk x,y,v\rsk\rsk$.
\end{compactitem}
\medskip
\end{definition}

\begin{theorem}\label{th: Jtern to quadr}Let $\Char(k)\neq2,3$. 
Let $J$ be the reduced spin factor of the non-degenerate anisotropic quadratic space $(k,V,q)$ with base point $u$. Let $X$ be a non-trivial $J$-ternary algebra such that $\lsk u\bt x,x\rsk \neq 0$ for all $x\in X_0\setminus\{0\}$.

Then $X$ satisfies \eqref{1evoorwaarde}. Therefore $(k,V,q,u,X_0,\cdot ,h)$ is a quadrangular algebra, with $\cdot$ and $h$ as in Theorem \ref{constr quad alg}.
\end{theorem}
\begin{proof}
Let $i\in\{0,1\}$, we will first show that for all $x\in X_i$ and $v\in J_\subhalf$ 
\begin{align}\label{charnot3}v\bt\lsk x,x,x\rsk=3\lsk v\bt x,x\rsk \bt x=3 \lsk v\bt x,x,x\rsk=\frac{3}{2} \lsk x,x,v\bt x\rsk.\end{align}

From (JT2) we find that $e_1\bt \lsk x, v\bt x,x\rsk=t \lsk x, v\bt x,x\rsk$ with $t=-1$ if $i=0$ and $t=2$ if $i=-1$, therefore
\[\lsk x, v\bt x,x\rsk=0.\]

Using (JT3) and (JT4) respectively we get
\[\lsk x,v\bt x\rsk \bt x=\lsk v\bt x,x,x\rsk-\lsk x,x,v\bt x\rsk\text{ and }\lsk x,v\bt x\rsk \bt x=-\lsk v\bt x,x,x\rsk.\]
Combining these equations, we obtain  
\[\lsk x,x,v\bt x\rsk=2\lsk v\bt x,x,x\rsk=2 \lsk v\bt x,x\rsk \bt x.\]
 From (JT2) we have 
\[v\bt\lsk x,x,x\rsk=\lsk v\bt x,x,x\rsk+\lsk x,x,v\bt x\rsk.\]

Combining the two last formulas proves \eqref{charnot3}.
Since $\Char(k)\neq3$, it follows from \eqref{charnot3} that \eqref{1evoorwaarde} is equivalent with 
\[v\bt\lsk x,x,x\rsk=v\bt( u\bt(u\bt\lsk x,x,x\rsk)).\]
Since this last equation holds we have proved that \eqref{1evoorwaarde} holds.
\end{proof}

\begin{remark}
\begin{compactenum}[(\rm i)]
\item We show that for all $x\in X_0$, $\lsk u\bt x,x\rsk=0 \iff \lsk x,x,x\rsk=0$.

First remark that if $v\bt x=0$ we have $v=0$ or $x=0$: it follows from $v\bt x=0$ that $v\bt(v\bt x)=q(v)x=0$, since $q$ is anisotropic we have $v=0$ or $x=0$. Now we have
\begin{align*}
&&\lsk u\bt x,x\rsk &=0\\
&\iff&\lsk u\bt x,x\rsk \bt x&=0\\
&\iff&u\bt \lsk x,x,x\rsk &=0\quad \text{by \eqref{charnot3}}\\
&\iff& \lsk x,x,x\rsk &=0.
\end{align*}
\item We have to demand that  $\lsk x,x,x\rsk \neq0$ for all $x\in X_0\setminus\{0\}$, because there exist $J$-ternary algebras which fulfill all the requirements but where $\lsk x,x,x\rsk =0$ for some $x\neq0\in X_0$. For example, consider \cite[Example 6.81]{ABG} with the zero skew hermitian form.  Examples like this we clearly want to avoid.
\end{compactenum}
\end{remark}

\subsection{Construction of quadrangular algebras of pseudo-quadratic form type}\label{sec:pseudo-quad}\label{sec33}
Let $k$ be a field of characteristic not 2.

We rely on the example \cite[6.81]{ABG} to obtain a quadrangular algebra of pseudo-quadratic form type using Theorem \ref{constr quad alg}. In combination with Section~\ref{sec:E6E7E8} this will show that all quadrangular algebras of characteristic not 2 can be obtained using the construction in Theorem \ref{constr quad alg}.

In Section~\ref{sec4} we will show that each Moufang quadrangle in characteristic not 2 can be obtained from a construction that generalizes the construction in Theorem \ref{constr quad alg}.

Let $(L/k,\si)$ be a quadratic pair with quadratic form $q(\ell)=\ell\ell^\si$ and let $(L,\si,X,h,\pi)$ be a standard anisotropic pseudo-quadratic space (see Section~\ref{def:pqs}), so $(k,L,q,u,X,\text{ scalar multiplication}, h)$ is a quadrangular algebra.

\begin{definition} 
\begin{compactenum}[(\rm i)]
\item Define $J=\mathcal{H}(M_2(L),\si)$ (see Definition \ref{def: matrix} and Remark \ref{mat quadr pair}); this Jordan algebra is a reduced spin factor of the quadratic form \[Q:J_\subhalf\rightarrow k:\begin{bmatrix}0&\ell^\si \\ \ell &0\end{bmatrix}\mapsto q(\ell).\]
As before we define $e_0=\begin{bmatrix}1&0\\0&0\end{bmatrix}, e_1=\begin{bmatrix}0&0\\0&1\end{bmatrix}, u=\begin{bmatrix}0&1\\1&0\end{bmatrix}\in J:$ $u$ is a base point of $Q$.

\item
Define $\widetilde{X}=X^2$, the $1\times 2$ row vectors over X. 

\item We define the action of $J$ on $\Xt$ as\footnote{On the right hand side the usual matrix multiplication is used.}  $j\bt x:=xj\in \widetilde{X}$ for $j\in J, x\in \widetilde{X}$. 

\item Define $\psi: \Xt\times\Xt\rightarrow M_2(L): \psi([x_1,x_2],[y_1,y_2]):=\big[ h(x_i,y_j)\big]$, now we define the skew product $\Xt\times \Xt\rightarrow J$ 
\begin{align*}
\lsk[x_1,x_2],[y_1,y_2]\rsk&:=\psi ([x_1,x_2],[y_1,y_2])-\psi([y_1,y_2],[x_1,x_2])\\
&=\begin{bmatrix} h(x_1,y_1)-h(y_1,x_1)& h(x_1,y_2)-h(y_1,x_2)\\ -h(y_2,x_1)+h(x_2,y_1)& h(x_2,y_2)-h(y_2,x_2) \end{bmatrix}. 
\end{align*}
\end{compactenum}
\end{definition}

With respect to $e_1$ we have \[\Xt_0=\{[x,0]\mid x\in X\}, \Xt_1=\{[0,x]\mid x\in X\}.\]

\begin{lemma}
The space $\Xt$ is a non-trivial special $J$-module with skew-symmetric bilinear form $\lsk \cdot , \cdot \rsk$ that satisfies \eqref{skew id}, \eqref{1evoorwaarde} and \eqref{2evoorwaarde} hold as well.

Under the identifications \[J_\subhalf\cong L: \ell\leftrightarrow \begin{bmatrix}0&\ell^\si\\ \ell &0\end{bmatrix}\text{ and }\Xt_0\cong X: [x,0]\leftrightarrow x,\] the quadrangular algebra defined in Theorem  \ref{constr quad alg} is exactly the pseudo-quadratic space we started with:
\[(k,L,q,u,X,\text{ scalar multiplication}, h)\]
\end{lemma}

\begin{proof}
Verifying that $\Xt$ is a special $J$-module and \eqref{skew id} requires some straightforward calculations. We will verify \eqref{1evoorwaarde} and \eqref{2evoorwaarde}.

Define $v=\begin{bmatrix}0&\ell^\si \\ \ell &0\end{bmatrix}\in J_\subhalf,\xt=[x,0]\in \Xt_0$. Notice that $u\bt \xt=[0,x]\in X_1$, and \[\lsk u\bt \xt,\xt\rsk=\begin{bmatrix}0&-h(x,x)\\ h(x,x) &0\end{bmatrix}.\] Hence $\lsk u\bt \xt,\xt\rsk$ is equal to $0$ if and only if $h(x,x)=0$. As we are working in a standard anisotropic pseudo-quadratic space, $\pi(x)=\half h(x,x)$ is anisotropic (see Remark \ref{charnot2}.2), so \eqref{2evoorwaarde}  holds.

 Condition \eqref{1evoorwaarde} holds since \[\lsk v\bt \xt,\xt\rsk \bt \xt=[0,-x h(x,x) l^\si]=v\bt [-x h(x,x),0 ]=v\bt( u\bt(\lsk u\bt \xt,\xt\rsk \bt \xt)).\]

From Theorem~\ref{constr quad alg} we conclude that $(k,J_\subhalf,Q,u,\Xt_0,\cdot ,\widetilde{h})$ is a quadrangular algebra with
\begin{align*}
&[x, 0]\cdot \begin{bmatrix}0&\ell^\si\\ \ell &0\end{bmatrix}=[0, x\ell],\\
&\widetilde{h}([x,0],[y,0])=\lsk u\bt[x,0],[y,0]\rsk= \begin{bmatrix}0&-h(y,x)\\ h(x,y) &0\end{bmatrix}.\tag*{\qedhere}
\end{align*}
\end{proof}

\subsection{Construction of quadrangular algebras of type $E_6,E_7, E_8$}\label{sec:E6E7E8}\label{sec34}

In this subsection we give a new construction of the vector spaces $X, L$ and the maps $\cdot$ and $h$ that we discussed in Section~\ref{se:e6e7e8}.

\subsubsection{A characterization of quadratic forms of type $E_6,E_7, E_8$}

We start by giving a new way to describe quadratic forms of type $E_6, E_7$ and $E_8$ (see  Definition \ref{def:e6e7e8}). The following illuminating observation was made by Skip Garibaldi. 

\begin{theorem}\label{th:e6e7e8}
Let $q$ be an anisotropic form over $k$ of dimension $6,8$ or $12$. Then $q$ is of type $E_6$, $E_7$ or $E_8$ respectively if and only if there exist an octonion division algebra $C_1$ and a division composition algebra $C_2$, of dimension $2,4$ or $8$ respectively, that have linkage number one such that $q$ is similar to the anisotropic part of the Albert form of $C_1\otimes_k C_2$. 
\end{theorem}

The `only if'-direction of this theorem is proved in the following, more technical, lemma.

\begin{lemma}\label{lem:e6e7e8}
We consider a quadratic form $q$ of type $E_6$, $E_7$ or $E_8$. Let $N$ denote the norm of a separable quadratic field extension $E=k(x)/(x^2-a)$ for $a\notin k^2$;
\begin{compactenum}[\rm (i)]
\item If $q=N\otimes \langle 1, s_2,s_3\rangle$ of type $E_6$, define \[C_1=(a,-s_2,-s_3)_k \text{ and } C_2=E.\]
\item If $q=N\otimes \langle 1, s_2,s_3,s_4\rangle$ of type $E_7$, define \[C_1=(a ,-s_2,-s_3)_k \text{ and } C_2=(a ,s_2s_3s_4)_k.\]
\item If $q=N\otimes \langle 1, s_2,s_3,s_4,s_5,s_6\rangle$ of type $E_8$, define 
\[C_1=(a ,-s_2,-s_3)_k\text{ and } C_2=(a ,-s_{4}s_{6},-s_{5}s_{6})_k.\]
\end{compactenum}
Then $q$ is similar to the anisotropic part of the Albert form of $C_1\otimes_kC_2$ and $C_1$ and $C_2$ are division algebras that have linkage number 1.
\end{lemma}
\begin{proof}
Denote the norm form of $C_1$ by $q_1$, the norm form of $C_2$ by $q_2$ and the Albert form of $C_1\otimes_k C_2$ by $q_A$. In the case that $q$ is of type $E_8$ we will verify that
\begin{align}\label{id:albertform}q\perp 2\Hh\sim q_A\perp \Hh\sim q_1\perp -q_2,\end{align}
the other two cases are similar. We have $q_1=N\otimes \llangle s_2,s_3\rrangle$ and $q_2= N\otimes \llangle s_{4}s_{6},s_{5}s_{6}\rrangle$. Since $t\Hh\simeq\Hh$ for $t\in k$ it follows that 
\begin{align}\label{constr type E8}
s_2s_3(\langle 1,s_2,s_3,s_4,s_5,s_6\rangle\perp \Hh)\notag&\simeq s_{2}s_{3} \langle s_3,s_2,1,s_5,s_4,s_6\rangle\perp \Hh\notag\\
&\simeq\langle s_2,s_3,s_2s_3,s_2s_3s_5,s_2s_3s_4,s_2s_3s_6\rangle\perp\Hh\notag\\
&\simeq\langle s_2,s_3,s_{2}s_{3},-s_{4}s_{6},-s_{5}s_{6},-s_{4}s_{5}\rangle\perp \langle 1,-1\rangle\notag\\
&\simeq\llangle s_2,s_3\rrangle \perp- \llangle s_{4}s_{6},s_{5}s_{6}\rrangle.
\end{align}
By multiplying the above identity with $N$ we obtain \eqref{id:albertform}.

Note that $q$ is anisotropic. Therefore $q_1\perp - q_2$ has Witt index 2; it follows that $q_1$ and $q_2$ are anisotropic and both $C_1$ and $C_2$ are  division algebras. It follows from \eqref{id:albertform} that $q_A$ has Witt index 1, and now Lemma \ref{lem:wittindex} implies that $C_1$ and $C_2$ have linkage number 1.
\end{proof}

\begin{proof}[Proof of Theorem \ref{th:e6e7e8}] The `only if'-direction is proven in the Lemma above. The `if'-direction follows in a similar way. We elaborate the case where $C_1$ and $C_2$ are octonion division algebras.

Since $C_1$ and $C_2$ contain an isomorphic field extension, by \cite[Prop. 1.5.1]{SV} we can assume that $C_1=(a,b_1,c_1)_k$ and $C_2=(a,b_2,c_2)_k$ for some $a,b_1,b_2,c_1,c_2\in k$. We denote the Albert form of $C_1\otimes_k C_2$ by $q_A$.

Define $N:=\llangle -a\rrangle$, this is anisotropic since $C_1$ is division. By going through \eqref{constr type E8} from bottom to top with \[s_2:=-b_1,s_3:=-c_1, s_4:=\frac{1}{b_1c_1c_2},s_5:=\frac{1}{b_1c_1b_2}, s_6:=-b_1b_2c_1c_2\] we find that $q_A$ is similar to $N\otimes \langle 1,s_2,s_3,s_4,s_5,s_6\rangle\perp \Hh$. Since the Witt index of $q_A$ is one, $N\otimes \langle 1,s_2,s_3,s_4,s_5,s_6\rangle$ is the anisotropic part of $q_A$; since $s_2s_3s_4s_5s_6=-1$ it is of type $E_8$.
\end{proof}

\subsubsection{The construction}

In order to construct quadrangular algebras of type $E_6, E_7$ and $E_8$ we follow Example 6.82 in \cite{ABG} closely. In {\em loc.\@ cit.\@} a $J$-ternary algebra is constructed out of $C_1\otimes_k C_2$ in \chacha\ zero, but this restriction is not necessary. 

First we give a motivation of the approach we will be following in our construction.

\begin{remark} Let $C_1$ be an octonion division algebra and $C_2$ a separable quadratic field extension, quaternion division algebra or octonion division algebra and assume that $C_1$ and $C_2$ have linkage number one. 

The dimension of $C_1\otimes_kC_2$ is $16$, $32$ or $64$, respectively. The space of skew elements is $\Ss=S_1\otimes 1+ 1\otimes S_2$ and has dimension $8$, $10$ or $12$, respectively. Let $(k,L,q,1,\widetilde{X},\cdot,h)$ be a quadrangular algebra of type $E_6$, $E_7$ or $E_8$, respectively. We summarize some dimensions:
\[
\begin{array}{c|c|c|c|}
& E_6 & E_7 & E_8 \\
 \hline
 \dim_k\Ss &8&10&14\\
 \dim_k L & 6&8&12\\
 \hline
\dim_k( C_1\otimes_kC_2 )&16&32&64\\
\dim_k \widetilde{X} & 8&16&32
\end{array}\]

We see that in all three cases $\dim_k\Ss=\dim_k L+2$ and $\dim_k(C_1\otimes C_2)=2\dim_k \widetilde{X}$. 

In Theorem~\ref{constr quad alg} we considered some objects the dimensions of which behave similarly: Let $J$ be a Jordan algebra of reduced spin type with base point and let $X$ be a special $J$-module. Then $\dim_k(J)=\dim J_\subhalf +2$ and $\dim_k X=2 \dim_k X_0$. 

From Lemma \ref{lem:e6e7e8}, the Albert form from $\Ss$ to $k$ can be written as the sum of a hyperbolic plane and a quadratic form of type $E_6, E_7$ or $E_8$, respectively. Note that a hyperbolic plane is two-dimensional.

In the following pages, we will give $\Ss$ the structure of a reduced spin factor of a quadratic form of type $E_6$, $E_7$ or $E_8$, respectively, and identify $J_\subhalf$ with $L$. Then we will give $C_1\otimes C_2$ the structure of a special $J$-module equipped with a bilinear skew-hermitian form, and identify $(C_1\otimes C_2)_0$ with $\widetilde{X}$.
\end{remark}
We start by fixing some notation.

\begin{notation}\begin{compactenum}[(\rm i)]
\item We fix a basis for the composition algebras $C_1$ and $C_2$ that have linkage number 1. We let $C_1$ be the octonion division algebra
\[C_1=\langle 1, \ii_1,\jj_1,\ii_1\jj_1,\kk_1,\ii_1\kk_1, \jj_1\kk_1,(\ii_1\jj_1)\kk_1\rangle.\]
If $C_2$ is a separable quadratic field extension, we define $C_2=\langle 1, \ii_2\rangle$. In the case $C_2$ is a quaternion division algebra we define $C_2=\langle 1, \ii_2,\jj_2,\ii_2\jj_2\rangle$. In the case $C_2$ is an octonion division algebra we define
 \[C_2=\langle 1, \ii_2,\jj_2,\ii_2\jj_2,\kk_2,\ii_2\kk_2, \jj_2\kk_2,(\ii_2\jj_2)\kk_2\rangle.\]
Since $C_1$ and $C_2$ have linkage number 1, we can choose these bases in such a way that \[\ii_1^2=\ii_2^2=:a\in K.\]

\item From now on we denote $\Ss_1$, $\Ss_2$ and  $\Ss$ for the set of skew elements of  $C_1$, $C_2$ and $C_1\otimes_k C_2$, respectively.

\item We denote the Albert form of $C_1\otimes_k C_2$ by $q_A:\Ss\rightarrow k$ and its associated bilinear form by $f_A$.
\item Let $V:=\langle \ii_1\otimes 1, 1\otimes \ii_2\rangle^\perp$ denote the orthogonal complement of the subspace $\langle \ii_1\otimes 1, 1\otimes \ii_2\rangle$ of $\Ss$ with respect to the non-degenerate bilinear form $f_A$. 
\end{compactenum}
\end{notation}

We want to make $\Ss$ into a Jordan algebra of reduced spin type. In particular it should contain supplementary proper idempotents $e_0$ and $e_1$ and an element $u\in J_\subhalf$ such that $u^2$ is the identity.
It will become clear that the elements constructed in the following lemma will be the ones we need.

\begin{lemma}\label{lem:e0e1u}
Let $u\in V\setminus\{0\}$ be arbitrary. Then up to scalars, there exists a unique pair $(e_0,e_1)$ of elements in $\Ss$ such that 
\begin{align*} &q_A(e_0)= q_A(e_1)=0,& \\&f_A(e_0,V)= f_A(e_1,V)=0,&\hspace{-1.5cm} f_A(e_0,e_1)=-q_A(u)\neq 0.  
  \end{align*}	
Explicitly, there exists an element $\la\in k$ such that $(\la e_0,\la^{-1}e_1)$ is equal to
\[\left(\ii_1\otimes 1 + 1\otimes \ii_2, \frac{q_A(u)}{4a}(\ii_1\otimes 1 - 1\otimes \ii_2)\right).\]
\end{lemma}
\begin{proof}
Since $ q_A$ has Witt index one, $q_A$ is anisotropic on $V=\langle \ii_1\otimes 1, 1\otimes \ii_2\rangle^\perp$. Hence $q_A(u)\neq 0$.

We demand that $e_0,e_1$ are isotropic elements in $V^\perp=\langle \ii_1\otimes 1, 1\otimes \ii_2\rangle$. This implies that they are of the form $\la(\ii_1\otimes 1 \pm 1\otimes \ii_2)$. 
Since $ f_A(e_0,e_1)$ should be different from $0$ we can take without loss of generality $e_0=\la_0(\ii_1\otimes 1 + 1\otimes \ii_2)$ and $e_1=\la_1(\ii_1\otimes 1 - 1\otimes \ii_2)$ for some $\la_0,\la_1\in k\setminus\{0\}$. 
Now we determine the scalars $\la_0,\la_1$, such that $ f_A(e_0,e_1)=-q_A(u)$. We have \[f_A(\la_0(\ii_1\otimes 1 + 1\otimes \ii_2),\la_1(\ii_1\otimes 1 - 1\otimes \ii_2))=\la_0\la_1(-4a).\]
So we find that $\la_1=\frac{q_A(u)}{4a \la_0}$.
\end{proof}

Since $\dim\Ss=\dim V+2$, we want to make $V$ into a quadratic space.
If we want that $u.u=1$ in the Jordan algebra of reduced spin type we will define, the element $u$ should be the base point of the quadratic form that determines the reduced spin factor. In the following definition we define a Jordan algebra on $\Ss$; in Lemma \ref{lem:LJ} we will show that this Jordan algebra has a natural interpretation in the endomorphism ring of $C_1\otimes_k C_2$.

\begin{definition}\label{def:Je0e1u} Let $u\in V\setminus\{0\}$ and 
\[e_0=\ii_1\otimes 1 + 1\otimes \ii_2, \quad e_1=\frac{q_A(u)}{4a}(\ii_1\otimes 1 - 1\otimes \ii_2).\]
\begin{compactenum}[(\rm i)]
\item We define a quadratic form on the vector space $V$, \[Q:=\frac{1}{q_A(u)}q_A|_V.\] We denote the corresponding bilinear form by $F$.

It follows from Theorem \ref{th:e6e7e8} that $(k,V,Q)$ is a quadratic space of type $E_6,E_7$ or $E_8$, respectively, with base point $u$.
\item\label{def:Je0e1u nr} We have $\Ss=ke_0\oplus V\oplus k e_1$, we define the Jordan multiplication as in Definition \ref{def: redspin}:
\begin{align*}
	(t_1e_i).(t_2 e_j) &= \delta_{ij}t_1t_2 e_i, \\
	(t e_i).v &= \half t v, \\
	v .w &= \half F(v,w)(e_0+e_1),
\end{align*}
for all $i,j \in \{0,1\}, v,w \in V, t,t_1,t_2\in k$. We denote this Jordan algebra by $J$, this is the reduced spin factor of $(k,V,Q)$.

\item\label{defr} We define 
\[r:=(e_0+e_1)^{-1}=-\frac{1}{ q_A(e_0+e_1)}(e_0+e_1)^\natural\in \Ss,\]
where the inverse and $\natural$ is as in Definition \ref{def:invsharp}. Notice that $e_0+e_1$ is the identity in the Jordan algebra $J$ on $\Ss$, the definition of $r$ has nothing to do with the inverse in $J$.
\item Let $s\in\Ss$, define $L_s\in \End_k(C_1\otimes_kC_2)$ as $L_sx:=sx$ for all $x\in C_1\otimes_kC_2$.
\end{compactenum}
\end{definition}

We will consider the Jordan algebra of the associative algebra $\End_k(C_1\otimes_k~C_2)$, denoted by $\End_k(C_1\otimes_k~C_2)^+$. 
We show that the algebra of reduced spin type we defined above, is isomorphic to a Jordan subalgebra of $\End_k(C_1\otimes_k~C_2)^+$. 

\begin{lemma}[{\cite[Example 6.82]{ABG}}]\label{lem:LJ}
Let $s_1,s_2\in \Ss$, we have
\[\frac{1}{2}(L_{s_1}L_rL_{s_2}L_r+L_{s_2}L_rL_{s_1}L_r)=L_{s_1. s_2}L_r,\]
 where $s_1. s_2$ denotes the multiplication in the algebra $J$ defined in Definition \ref{def:Je0e1u}.\eqref{def:Je0e1u nr}.

Therefore $L_\Ss L_r$ is a Jordan subalgebra of $\End_k(C_1\otimes_k C_2)^+$ isomorphic to $J$.
\end{lemma}
\begin{proof}
We will make use of \cite[Proposition 3.3 (3.8)]{A}. In \cite{A} only characteristic 0 is considered; however this proposition can be generalized to characteristic different from 2 without any adjustments. The proof of this proposition uses basic identities of octonions (see Lemma \ref{identities octonion}) and the identity $s_1(s_2(s_1x))=~(s_1s_2s_1)x$ for $x\in C_1\otimes_kC_2$ (see Lemma \ref{ident bioctonion}).

Linearizing \cite[Prop 3.3 (3.8)]{A} gives
\begin{multline*}
	L_{s_1}L_{(e_0+e_1)^\natural}L_{s_2}+L_{s_2}L_{(e_0+e_1)^\natural} L_{s_1} \\
	= - f_A({s_1},e_0+e_1)L_{s_2}- f_A({s_2},e_0+e_1)L_{s_1}+ f_A({s_1},{s_2})L_{e_0+e_1}.
\end{multline*}

Since $r=(e_0+e_1)^{-1}=-\frac{1}{ q_A(e_0+e_1)}(e_0+e_1)^\natural=\frac{1}{ q_A(u)}(e_0+e_1)^\natural$, we find that 
\begin{multline*}
\frac{1}{2}(L_{s_1}L_rL_{s_2}L_r+L_{s_2}L_rL_{s_1}L_r)\\
=\frac{1}{2q_A(u)}(- f_A({s_1},e_0+e_1)L_{s_2}L_r- f_A({s_2},e_0+e_1)L_{s_1}L_r+ f_A({s_1},{s_2})L_{e_0+e_1}L_r).
\end{multline*}

 It follows from $u\in V=\langle e_0,e_1\rangle ^\perp$, $ q_A(e_0)= q_A(e_1)=0,  f_A(e_0,e_1)=- q_A(u)$, that for $i,j\in \{0,1\}$ and for all $v,w\in V$ 
\begin{align*}
\half(L_{e_i}L_rL_{e_j}L_r+L_{e_j}L_rL_{e_i}L_r)&=\delta_{ij}L_{e_i}L_r,\\
\half(L_{e_i}L_rL_{v}L_r+L_{v}L_rL_{e_i}L_r)&=\half L_{v}L_r,\\
\half(L_{v}L_rL_{w}L_r+L_{w}L_rL_vL_r)&=\frac{ f_A(v,w)}{2 q_A(u)}L_{e_0+e_1}L_r.
\end{align*}
This is exactly the multiplication of $J$.
\end{proof}

In order to define an action of $J$ on $C_1\otimes_k C_2$, we use the isomorphism of the previous Lemma.

\begin{definition}\label{def:bil}
We define the bilinear action
\[\bt: \Ss\times (C_1\otimes_k C_2)\rightarrow C_1\otimes_k C_2:(s,x)\mapsto L_sL_rx=s(rx).\]

We define the {\em skew symmetric bilinear map} \[\lsk.,.\rsk:(C_1\otimes_k C_2) \times (C_1\otimes_k C_2)\rightarrow \Ss:(x,y)\mapsto x\overline{y}-y\overline{x}.\]
\end{definition}

\begin{remark}\label{expl form X0}
\begin{compactenum}[(\rm i)]
\item After some computation we find 
\begin{align*} e_0\bt \left( x_1\otimes x_2\right)&=\half\left(x_1\otimes x_2+\frac{1}{a} \ii_1x_1\otimes \ii_2 x_2\right),\\
 e_1\bt \left(x_1\otimes x_2\right)&=\half\left(x_1\otimes x_2-\frac{1}{a}\ii_1x_1\otimes \ii_2x_2\right),
 \end{align*}
 for all $x_1\in C_1,x_2\in C_2$. Note that this is independent of the choice of the base point $u$.

\item\label{ii} Let $C$ be a composition algebra and define the bilinear skew symmetric map \[\psi:C\times C\rightarrow\Ss: (x,y)\mapsto x\overline{y}-y\overline{x}.\] Then for all $x_1, y_1\in C_1,x_2,y_2\in C_2$ we have,
\[\lsk x_1\otimes x_2,y_1\otimes y_2\rsk= f_2(x_2,y_2)\psi(x_1,y_1)\otimes 1+1\otimes f_1(x_1,y_1)\psi(x_2,y_2).\]

\end{compactenum}

\end{remark}

In the following theorem we show that the construction given in the introduction does indeed give rise to a quadrangular algebra of type $E_6,E_7$ or $E_8$. In the proof we make a distinction between the cases $\cha(k)\neq2$ and $\cha(k)\neq2,3$. When $\cha(k)\neq2,3$, $C_1\otimes_kC_2$ is a structurable algebra and we can use the theory of structurable algebras. 

If $\cha(k)=3$ we can not make use of the theory of structurable algebras; therefore we prove this in a direct way only making use of identities in octonions. Regrettably, this gives rise to lengthy computations and for one particular identity we had to rely on the computer algebra software \cite{sage}. This proof does not use the fact that the \chacha\ is equal to 3, but only that it is different from 2.

\begin{theorem}\label{th:E6E7E8} Let $\Char(k)\neq 2$.
Let $e_0,e_1,u\in \Ss$ be as in Lemma \ref{lem:e0e1u}, let the quadratic form $Q$ of type $E_6,E_7,E_8$ and the reduced spin factor $J$ be as in Definition \ref{def:Je0e1u}. Let  $X:=C_1\otimes_k C_2$, let $\bt$ and $\lsk.,.\rsk$ be defined as above. 

Then $X$ is a special $J$-module and  $\lsk. ,.\rsk$ satisfies each side of \eqref{skew id}. Conditions \eqref{1evoorwaarde} and \eqref{2evoorwaarde} of Theorem \ref{constr quad alg} are satisfied. 
As in Theorem \ref{constr quad alg} we define 
\begin{align*}
\cdot&:X_0\times V\rightarrow X_0: x\cdot v=v\bt(u\bt x)\\
h&:X_0\times X_0\rightarrow V:(x,y)\mapsto\lsk u\bt x,y\rsk.
\end{align*}

Then $(k,V,Q,u,X_0,\cdot ,h)$ is a quadrangular algebra of type $E_6,E_7,E_8$.
\end{theorem}

\begin{proof}We have from Lemma \ref{ident bioctonion}, that $(e_0+e_1)\bt x=x$ for all $x\in X$.  The fact that $X$ is a special $J$-module now follows from Lemma \ref{lem:LJ}. It follows from Theorem \ref{th:class quad alg} that if  $(k,V,Q,u,X_0,\cdot ,h)$ is a quadrangular algebra, it has to be of type $E_6,E_7,E_8$ due to the dimension of $V$.
\subsubsection*{\boldmath $\cha(k)\neq 2,3$}
In $\cha(k)\neq 2,3$ we can use the theory of structurable algebras to prove \eqref{skew id} and \eqref{1evoorwaarde}.

In \cite[Remark 6.7]{ABG} it is pointed out that each structurable algebra $\A$ (in our case $C_1\otimes_k C_2$) with an invertible skew element is a $J$-ternary algebra with $J=L_\Ss L_r\subset \End_k(\A)^+$. The action of the Jordan algebra on $X$ and the skew bilinear map are defined as in Definition~\ref{def:bil} above; the trilinear product is defined as
\[X\times X\times X\rightarrow X:(x,y,z)\mapsto -V_{x,ry}z:=(x(\overline{y}r))z+(z(\overline{y}r))x+(z\overline{x})(ry).\]

 \cite{ABG} only considers fields of characteristic 0. We checked that every structurable algebra, with an invertible skew element, is a $J$-ternary algebra, in characteristic different from 2 and 3. This proof is omitted in \cite[Remark 6.7]{ABG} and uses deep identities in structurable algebras. We thank Bruce Allison for giving us a detailed explanation on how to prove this fact. 

For the proof of \eqref{2evoorwaarde} we refer to the general characteristic case below. It now follows from Theorem \ref{th: Jtern to quadr} that $(k,V,Q,u,X_0,\cdot ,h)$ is a quadrangular algebra.

\subsubsection*{\boldmath $\cha(k)\neq 2$}
We first verify that the identity in the right hand side of \eqref{skew id} holds, this takes a rather lengthy but straightforward computation:

Since the condition is linear in $x$ and $y$, one can choose $x=x_1\otimes x_2$ and $y=y_1\otimes y_2$ for $x_1,y_1\in C_1, x_2, y_2\in C_2$. Let $s=s_1\otimes 1+1\otimes s_2\in \Ss$ and denote $r=r_1\otimes 1+1\otimes r_2$, instead of using its definition with coordinates. Using Remark \ref{expl form X0}.\eqref{ii} it is not hard to show that the following identities hold for $i\in \{1,2\}$\begin{compactitem}
\item $ f_i(s_i,\psi(x_i,y_i))=-2f_i(s_ix_i,y_i)$,
\item $\psi(s_ix_i,y_i)+\psi(s_iy_i,x_i)=2s_i f_i(x_i,y_i)$,
\item $f_i(s_i(r_ix_i),y_i)+ f_i(s_i(r_iy_i),x_i)=-f_i(s_i,r_i)f_i(x_i,y_i)$.
\end{compactitem}
Using these identities, \eqref{skew id} can be simplified to
\begin{multline*}\psi(s_i(r_ix_i),y_i)-\psi(s_i(r_iy_i),x_i)\\=-f_i(r_i,\psi(x_i,y_i))s_i+f_i(s_i,\psi(x_i,y_i))r_i-f_i(s_i,r_i)\psi(x_i,y_i),\end{multline*}
and this identity can be checked using  Lemma \ref{identities octonion}, especially  the Moufang identities \eqref{moufang id}.

We were not able to verify  \eqref{1evoorwaarde} by hand. The problem is that \eqref{1evoorwaarde} has degree 3 in $x$, so we can not assume that $x$ is of the form $e_0\bt (x_1\otimes x_2)$. 
We did a computation based on a coordinatization of $X$,  we used the software \cite{sage} to do the symbolic computations.

Now $x$ is an arbitrary element in $X_0=e_0\bt X, $ therefore $x$ is a sum of elements of the form $x_1\otimes x_2+\frac{1}{a} \ii_1x_1\otimes \ii_2 x_2$ (see Remark \ref{expl form X0}). We implemented octonions and the tensor product of two octonions in Sage in a symbolic way, and we verified that \eqref{1evoorwaarde} holds.

The only fact that remains to be verified is \eqref{2evoorwaarde}. In fact, this is exactly axiom (D2) and in the proof of Theorem \ref{constr quad alg} the condition \eqref{2evoorwaarde} is not used to prove any of the other axioms. Since we already know that the axioms A-B-C-D1 are true, we will use these to prove $\lsk u\bt x,x\rsk \neq0$ for all $x\in X_{0}\setminus\{0\}$.

First we show that 
\begin{align}\label{anisX0}
\text{there exists an }x\in X_0\text{ such that }\lsk u\bt x,x\rsk \neq0.\end{align}
 Let $x=e_0\bt (x_1\otimes x_2)\in X_0, u=s_1\otimes 1+1\otimes s_2 \in V$, with some calculation using Lemma \ref{identities octonion}, Remark \ref{expl form X0} and the coordinate expression for $r$ we find that
\[\lsk u\bt x,x\rsk=\frac{1}{4a}(q_2(x_2)\psi(s_1x_1,\ii_1x_1)\otimes 1+1\otimes q_1(x_1)\psi(s_2x_2,\ii_2x_2)).\]

Since $C_1$ and $C_2$ are division algebras, it is enough to show that for all $y\in C_1\minuszero$ we have $\psi(s_1y,\ii_1y)\neq 0$. We assume that $y\neq 0$ and $\psi(s_1y,\ii_1y)= 0$ and deduce a contradiction.
\begin{align*}
&&\psi(s_1y,\ii_1y)&= 0\\
&\Rightarrow& (s_1y)(\overline{y}\ii_1)-(\ii_1y)(\overline{y}s_1)&=0\\
&\Rightarrow& 2(s_1y)(\overline{y}\ii_1)&=f_1(\ii_1y,s_1y)1&\text{ since } y\overline{z}+z\overline{y}=f_1(y,z)1\\
&\Rightarrow& s_1y&=\half f_1(\ii_1y,s_1y)(\overline{y}\ii_1)^{-1}& \text{ since }\overline{y}\ii_1\neq 0\\
&\Rightarrow& s_1y&=-\frac{f_1(\ii_1y,s_1y)}{2 q_1(\overline{y} \ii_1)} \ii_1 y&\text{ since }y^{-1}=\frac{1}{q_1(y)}\overline{y}\\
&\Rightarrow& s_1&=-\frac{f_1(\ii_1y,s_1y)}{2 q_1(\overline{y} \ii_1)} \ii_1.
\end{align*}
This is a contradiction since $s_1\perp \ii_1$.

The rest of the proof is inspired by the proof given in \cite[Theorem 13.47]{TW}.

We fix an arbitrary $x\neq 0\in X_0$, notice that we no longer assume that $x$ has the form $e_0\bt (x_1\otimes x_2)$. We suppose that $\lsk u\bt x,x\rsk=0$ and aim to get a contradiction. It follows \footnote{ Since $q_A$ is anisotropic on $V:$ $v\bt x=0\iff v\bt v\bt x=q_A(v) x=0\iff v=0$ or $x=0$.} from \eqref{1evoorwaarde} that $\lsk v\bt x,x\rsk=0$ for all $v\in V$.

We first show that there exists an element $y\in X_0$ such that $\lsk u\bt x,y\rsk\neq0$. Suppose that $\lsk u\bt x, X_0\rsk=0$. It follows from (B2) that for all $y\in X_0, v\in V$
\begin{align*}
\lsk u\bt x,v\bt(u\bt y)\rsk&=\lsk u\bt y,v\bt(u\bt x)\rsk\\
&=-\lsk v\bt(u\bt x), u\bt y\rsk\\
&=-U_u \lsk u\bt(v\bt(u\bt x)), y\rsk
\end{align*}
Therefore $\lsk u\bt(x\cdot v),X_0\rsk =0$ and by repeating this procedure we obtain \[\lsk u\bt(x\cdot C(Q,u)), X_0\rsk=0.\] From Definition \ref{def:irmodule}.\eqref{cq1module} and Theorem \ref{th:irmodule} it follows that $X_0$ is an irreducible $C(Q,u)$-module, therefore we obtain $\lsk u\bt X_0, X_0\rsk=0$. This contradicts \eqref{anisX0}.

From now on we assume that $y\in X_0$ is such that $\lsk u\bt x,y\rsk\neq0$. Next we show that
\begin{align} ( x\cdot \lsk u\bt x,y\rsk)\cdot v= x\cdot \lsk u\bt x, y\cdot v\rsk.\label{eq:anis1}
\end{align}

Since this identity is trivial for $u=v$, we assume $v\perp u$. Then \eqref{eq:anis1} is equivalent to
\begin{align}
&\iff&v\bt u\bt \lsk u\bt x,y\rsk\bt u\bt x&= \lsk u\bt x, v\bt u\bt y\rsk\bt u\bt x\notag\\
&\iff&v\bt U_u\lsk u\bt x,y\rsk\bt x&= -\lsk u\bt x, u\bt v\bt y\rsk\bt u\bt x\notag\\
&\iff&v\bt\lsk x, u\bt y\rsk\bt x&= -U_u\lsk  x,  v\bt y\rsk\bt u\bt x\notag\\
&\iff&v\bt\lsk x, u\bt y\rsk\bt x&= -u\bt\lsk  x,  v\bt y\rsk\bt x\notag\\
&\iff& v\bt u\bt \lsk u\bt y,x\rsk\bt x&=  \lsk v\bt y, x\rsk \bt x.\label{eq:anis2}
\end{align}

We consider  \eqref{1evoorwaarde} for $y+tx$ for a parameter $t\in k$, we compare the terms that have degree one in $t$ using the assumption that $\lsk v\bt x,x\rsk=0$ for all $v\in V$, and we get
\begin{align*}
&&\lsk v\bt x,y\rsk \bt x+\lsk v\bt y,x\rsk \bt  x&=v\bt u\bt(\lsk u\bt x,y\rsk+\lsk u\bt y,x\rsk) \bt x\\
&\iff& 2\lsk v\bt y,x\rsk \bt  x+2(v \lsk x,y\rsk) \bt x&=2v\bt u\bt(\lsk u\bt y,x\rsk+(u\lsk x,y\rsk)) \bt x\\
&\iff& \lsk v\bt y,x\rsk \bt  x&=v\bt u\bt\lsk u\bt y,x\rsk \bt x,
\end{align*}
since $v \lsk x,y\rsk\in J_1$ and $x\in X_0$. This proves \eqref{eq:anis2}.

From \eqref{eq:anis1} we have $(x\cdot \lsk u\bt x,y\rsk)\cdot V \subseteq x\cdot V$; since $\lsk u\bt x,y\rsk\neq 0$, the dimension of those two vector spaces is equal and we find that 
\begin{align}\label{eq:anis3}
(x\cdot \lsk u\bt x,y\rsk)\cdot V =x\cdot \lsk u\bt x, y\cdot V\rsk= x\cdot V.
\end{align} 
For arbitrary $w\in V$ it follows from \eqref{eq:anis1} that 
\[( x\cdot \lsk u\bt x,y\cdot w\rsk)\cdot v= x\cdot \lsk u\bt x, (y\cdot w)\cdot v\rsk\in x\cdot V.\]
From \eqref{eq:anis3} we find that $(x\cdot V)\cdot V= x\cdot V$ and hence 
\[x\cdot C(q,u)=x\cdot V\neq X_0\]
contradicting the irreducibility of $X_0$. This finishes the proof of Theorem \ref{th:E6E7E8}.
\end{proof}

\begin{remark}
The map $g:X_0\times X_0\rightarrow k:(x,y)\mapsto \half f(h(y,x),1)$ takes an elegant expression. Indeed,
\begin{align*}
g(x,y)e_0&=\half f(\lsk u\bt y,x\rsk ,u)e_0\\
&=(\lsk u\bt y,x\rsk u)e_0\\
&=\half (\lsk y,x\rsk+\lsk u\bt y,u\bt x\rsk)e_0.
\end{align*}
Since $\lsk y,x\rsk\in k e_0$ and $\lsk u\bt y,u\bt x\rsk\in ke_1$, we conclude that $g(x,y)e_0=\half \lsk y, x\rsk$. When
we identify $k$ and $ke_0$, we have
\[g(x,y)=\half \lsk y, x\rsk.\]
\end{remark}
It follows from the previous theorem that we have, in characteristic not~2, a new coordinate-free definition of the various maps introduced in \cite[Chapter 13]{TW}.

\begin{remark}
The reader might wonder what will happen if we apply our construction in the case that both $C_1$ and $C_2$ are composition algebras of dimension $2$ or $4$ with mutual linkage number $1$.
In the three different cases that arise in this way, we get the following dimensions for the different relevant vector spaces.
\[
\begin{array}{c|c|c|c|}
& E \otimes E & E \otimes Q & Q_1 \otimes Q_2 \\
 \hline
 \dim_k\Ss &2&4&6\\
 \dim_k L & 0&2&4\\
 \hline
\dim_k( C_1\otimes_kC_2 )&4&8&16\\
\dim_k \widetilde{X} & 2&4&8
\end{array}\]
In the first case, the vector space $L$ is trivial, so our construction no longer applies (we cannot find an element $u \in V \setminus \{ 0 \}$
needed in Definition~\ref{def:Je0e1u}).

In the second case, the space $\tilde X$ gets the structure of a $2$-dimensional vector space over $E$, and the corresponding quadrangular algebra
is isomorphic to a quadrangular algebra of pseudo-quadratic form type with underlying vector space $\tilde X$.
Notice that $E \otimes Q \cong M_2(E)$ since $E$ and $Q$ are $1$-linked.

Similarly, in the third case, the space $\tilde X$ gets the structure of a $2$-dimensional vector space over a quaternion division algebra $Q_3$, and the corresponding quadrangular algebra is isomorphic to a quadrangular algebra of pseudo-quadratic form type with underlying vector space $\tilde X$.
The algebra $Q_3$ is the quaternion algebra with norm form similar to $q_A$, and in this case $Q_1 \otimes Q_2 \cong M_2(Q_3)$.
\end{remark}

\subsubsection{A final remark}

In an earlier paper \cite{BD}, we had found another related but quite different class of structurable algebras that seems to play an important role in
the understanding of the exceptional Moufang quadrangles of type $E_6$. $E_7$ and $E_8$.
That structurable algebra is, in each case, obtained by {\em doubling} another algebra (instead of {\em halving} an algebra as we did in the current paper).

More precisely, to each Moufang quadrangle $\Omega$ of type $E_6$, $E_7$ or $E_8$, we can associate a structurable algebra $Y$, the isotopy class of which is a complete invariant of the Moufang quadrangle $\Omega$, and which is obtained by applying the so-called Cayley--Dickson doubling process on the Jordan algebra $A^+$, where
\begin{compactenum}[(i)]
\item $A$ is a quaternion algebra $Q$ if $\Omega$ is of type $E_6$;
\item $A$ is a tensor product $Q \otimes E$ with $Q$ a quaternion algebra and $E$ a quadratic extension, if $\Omega$ is of type $E_7$;
\item $A$ is a biquaternion algebra $Q_1 \otimes Q_2$ if $\Omega$ is of type $E_8$.
\end{compactenum}
However, we are not yet aware of a direct way of relating the structurable algebra $Y$ with the structurable algebra $X = C_1 \otimes C_2$ which we have investigated in the current paper.

\section{A unified construction for Moufang quadrangles in characteristic not 2}\label{se:quadrangularsystems} \label{sec4}

\subsection{Preliminaries on Moufang quadrangles}
A {\em Moufang polygon} is a notion from incidence geometry introduced by Jacques Tits.
We only give a brief summary of the theory of Moufang quadrangles, and we refer to \cite{TW} for more details.
The importance will immediately become clear in Theorem~\ref{th:mpol} below.

A {\em generalized quadrangle} $\Gamma$ is a connected bipartite graph with diameter $4$ and girth $8$.
We call a generalized polygon {\em thick} if every vertex has at least $3$ neighbors.
A {\em root} in $\Gamma$ is a (non-stammering) path of length $4$ in $\Gamma$.

Let $\Gamma$ be a thick generalized quadrangle, and let $\alpha = (x_0,\dots,x_4)$ be a root of $\Gamma$.
Then the group $U_\alpha$ of all automorphisms of $\Gamma$ fixing all neighbors of $x_1, x_2, x_{3}$ (called a {\em root group})
acts freely on the set of vertices incident with $x_0$ but different from $x_1$.
If $U_\alpha$ acts transitively on this set (and hence regularly), then we say that $\alpha$ is a {\em Moufang root}.

A {\em Moufang quadrangle} is a generalized quadrangle for which every root is Moufang.
We then also say that $\Gamma$ satisfies the {\em Moufang condition}.

Moufang quadrangles have been classified by J.~Tits and R.~Weiss \cite{TW}.
Loosely speaking, the result is the following.
\begin{theorem}[\cite{TW}]\label{th:mpol}
	Every Moufang quadrangle arises from an absolutely simple linear algebraic group of relative rank $2$,
	or from a corresponding classical group or group of mixed type.
\end{theorem}
In particular, every Moufang quadrangle is of ``algebraic origin'', and in fact, the Moufang quadrangles provide a useful tool
to help in the understanding of the corresponding groups; this is particularly true for the Moufang quadrangles arising
from linear algebraic groups of exceptional type.
For instance, the Kneser--Tits problem for groups of type $E_{8,2}^{66}$ has recently been solved using the theory of Moufang polygons \cite{PTW}.

In order to describe a Moufang quadrangle in terms of algebraic data, we will use so-called {\em root group sequences}.
A root group sequence for a Moufang quadrangle is a sequence of $4$ root groups, labeled $U_1,\dots,U_4$, together with
{\em commutator relations} describing how elements of two different root groups $U_i$ and $U_j$ commute.
In each case, the commutator of an element of $U_i$ and $U_j$ (with $i<j$) belongs to the group $\langle U_{i+1},\dots, U_{j-1} \rangle$.
The following result is crucial.
\begin{theorem}\label{th: class mouf}
	Let $\Gamma$ be a Moufang quadrangle.
	Then $\Gamma$ is completely determined by the root groups $U_1,\dots,U_4$ together with their commutator relations.
\end{theorem}
\begin{proof}
	See \cite[Chapter 7]{TW}.
\end{proof}
For more details about this procedure, and how the Moufang polygons can be reconstructed from the root group sequences,
we refer to \cite{TW} or to the survey article \cite{DV}.

For each type of Moufang quadrangle, we will describe an {\em algebraic structure} which will allow us to parametrize the root groups and describe the commutator relations.

In principle, it is possible to define a single algebraic structure to describe all possible Moufang quadrangles;
this gives rise to the so-called {\em quadrangular systems} which have been introduced by the second author \cite{D}.
These structures, however, have some disadvantages from an algebraic point of view; most notably, the definition does not
mention an underlying field of definition (although it is possible to construct such a field from the data), and
the axiom system looks very wild and complicated, with no less than $20$ defining identities.

Here follows the original classification as given by Tits and Weiss in \cite{TW}, distinguishing six different
(non-disjoint) classes:
\begin{compactenum}[(1)]
\item Moufang quadrangles of indifferent type;
\item Moufang quadrangles of quadratic form type;
\item Moufang quadrangles of involutory type;
\item Moufang quadrangles of pseudo-quadratic form type;
\item Moufang quadrangles of type $E_6$, $E_7$ and $E_8$;
\item Moufang quadrangles of type $F_4$.
\end{compactenum}
The Moufang quadrangles of types (2)--(4) are often called {\em classical}, those of type (5) and (6) are called {\em exceptional}
and those of type (1) are of {\em mixed type}.
Since the Moufang quadrangles of type (1) and (6) only exist over fields of characteristic two, and moreover are not directly related
to rank two forms of algebraic groups, we exclude those two classes from our further discussion.

In the following section we will give a uniform description of the remaining 4 classes of Moufang quadrangles, over fields of characteristic different from 2,  starting from a special Jordan module.

\subsection{Construction of Moufang quadrangles from special Jordan modules}
We will show that each type of Moufang quadrangle in characteristic not 2 can be described in a unified way from a special $J$-module. We generalize the procedure that we used in Theorem \ref{constr quad alg} to obtain quadrangular algebras. In order to obtain all Moufang quadrangles we allow that $\dim(J_0)>1$ and we allow the special $J$-module to be the trivial module.
It follows from Theorem~\ref{th: class mouf} that it is sufficient to describe the $4$ root groups and the commutator relations of the root groups to describe the Moufang quadrangle completely.

\begin{construction}\label{mouf quadr}
Let $J$ be a non-degenerate Jordan algebra that contains supplementary proper idempotents $e_0$ and $e_1$. Let $J_0,J_\subhalf,J_1$ be the Peirce subspaces of $J$ with respect to $e_1$. We assume that each element in $J_\subhalf\setminus\{0\}$ is invertible and that there exists $u\in J_\subhalf$ such that $u^2=1$.

Let $X$ be a special $J$-module equipped with a skew-symmetric bilinear form $\lsk \cdot , \cdot \rsk:X\times X\rightarrow J$. 
\begin{compactitem}
\item Define the abelian group $V:=(J_\subhalf,+)$.
\item Define the (not necessary abelian) group $W:=X_0\times J_0$ with addition \[[a_1,t_1]\boxplus[a_2,t_2]=[a_1+a_2,t_1+t_2+\half \lsk a_2,a_1\rsk].\]

Notice that the inverse is $\boxminus[a,t]=[-a,-t]$.
\end{compactitem}

Let $U_1$ and $U_3$ be two groups isomorphic to $W$, and
let $U_2$ and $U_4$ be two groups isomorphic to $V$.
Denote the corresponding isomorphisms by
\begin{center}
	\begin{tabular}{r@{ : }c@{ : }l@{ ;}l}
		$x_1$ & $W \to U_1$  & $[a,t] \mapsto x_1(a,t)$ & \\
		$x_2$ & $V \to U_2$       & $v \mapsto x_2(v)$ & \\
		$x_3$ & $W \to U_3$  & $[a,t] \mapsto x_3(a,t)$ & \\
		$x_4$ & $V \to U_4$       & $v \mapsto x_4(v)$ & 
	\end{tabular}
\end{center}
we say that $U_1$ and $U_3$ are {\em parametrized} by $W$ and that $U_2$ and $U_4$ are
{\em parametrized} by $V$.

Now, we implicitly define the group $U_+ = \langle U_1,U_2,U_3,U_4\rangle$ by the following
commutator relations:
\begin{align*}
	[x_1(a_1,t_1), x_3(a_2,t_2)^{-1}] &= x_2\bigl(\lsk u\bt a_1,a_2\rsk\bigr) \ , \\
	[x_2(v_1), x_4(v_2)^{-1}] &= x_3\bigl(0,2(v_1v_2)e_0\bigr) \ , \\
	[x_1(a,t), x_4(v)^{-1}] &= x_2\bigl(\half \lsk u\bt a, v\bt(u\bt a)\rsk+2 (U_ut)v\bigr) x_3\bigl(v\bt(u\bt a),U_vU_ut\bigr) \ , \\
	[U_i, U_{i+1}] &= 1 \quad \forall i \in \{ 1, 2, 3 \} \ ,
\end{align*}
for all $[a,t], [a_1,t_1], [a_2,t_2] \in W$ and all $v, v_1, v_2 \in V$.
\end{construction}

It follows from Lemma \ref{char red spin} that $J$ should be either of reduced spin type or of type $\mathcal{H}(M_2(L),\si t)$.
For each of these two cases, we will distinguish between the zero special $J$-module and a non-zero $J$-module. Case by case, we will show that in this way the root groups $U_1,U_2,U_3,U_4$ and commutation relations given above coincide with the description given in Chapter 16 of \cite{TW} of the Moufang quadrangles in characteristic not 2. 

\begin{remark}
In \cite{Dmem} quadrangular systems are introduced. These are structures that are defined by 24 axioms, which describe in a unified way all Moufang quadrangles (including \chacha\ 2.) We believe it should be possible to start with Construction \ref{mouf quadr}, impose a few more axioms that look like the ones in Theorem \ref{constr quad alg} and prove all the axioms defining a quadrangular system. However the verifications of the axioms that use the map $\kappa$, this is a kind of  ``multiplicative inverse" in the group $W$, get very complicated.
\end{remark}

\paragraph{Moufang quadrangles of quadratic form type}

Let $J$ be a reduced spin factor of an anisotropic, non-degenerate quadratic space $(k,V,q)$ with base point $u$. Let $X$ be the zero module over $J$.

Remember that $J_0=ke_0,J_\subhalf=V, J_1=ke_1$. 
\begin{compactitem}
\item Define the abelian group $V=(J_\subhalf,+)$.
\item Define the group $W=X_0\times J_0=\{[0,t e_0]|t \in k\}\cong k$ with addition $[0,t_1 e_0]\boxplus[0,t_2 e_0]=[0,(t_1+t_2)e_0]$. Therefore $W$ is isomorphic to the additive group of $k$ with corresponding isomorphism $W\cong k:[0,te_0]\leftrightarrow t$. So we will write $x_1(t):=x_1(0,te_0)$ and $x_3(t):=x_3(0,te_0)$.
\end{compactitem}
Let $U_1$ and $U_3$ be parametrized by $W$ and $U_2$ and $U_4$ be parametrized by $V$. Let $t,t_1,t_2\in k$, $v,v_1,v_2\in V$; using the formulas for the multiplication and the $U$-operator in a Jordan algebra of reduced spin type (see Definition \ref{def: redspin}) we find for the commutator relations
\begin{align*}
	[x_1(t_1), x_3(t_2)^{-1}] &=[x_1(0,t_1e_0), x_3(0,t_2e_0)^{-1}]= x_2\bigl(0)=1 \ , \\
	[x_2(v_1), x_4(v_2)^{-1}] &= x_3\bigl(0,f(v_1,v_2)e_0\bigr)=x_3\bigl(f(v_1,v_2)\bigr) \ , \\
	[x_1(t), x_4(v)^{-1}] &=[x_1(0,te_0), x_4(v)^{-1}]= x_2\bigl(2 (U_ut e_0)v\bigr) x_3\bigl(0,U_vU_ut e_0\bigr)
	\\&=  x_2\bigl(2 (t e_1)v\bigr) x_3\bigl(0,U_vt e_1\bigr)= x_2\bigl( t v\bigr) x_3\bigl(0,q(v)t e_0\bigr)\\
	&=x_2\bigl( t v\bigr) x_3\bigl(q(v)t\bigr)\ , \\
	[U_i, U_{i+1}] &= 1 \quad \forall i \in \{ 1, 2, 3 \} \ .
\end{align*}

We obtain exactly the  same description as in \cite[Example 16.3]{TW}.

If $d = \dim_K V$ is finite, then these Moufang quadrangles arise from linear algebraic groups;
they are of absolute type $\mathsf{B}_{\ell+2}$ if $d = 2\ell + 1$ is odd, and of type $\mathsf{D}_{\ell+2}$ if $d = 2\ell$ is even.

\paragraph{Moufang quadrangles of type $E_6,E_7,E_8$}

This case was actually already handled in Theorem \ref{th:E6E7E8}, since from quadrangular algebras one can define the root groups and commutation relations of the corresponding Moufang quadrangles, see \cite[Chapter 11]{W}. Now we quickly verify that we get indeed the right commutator relations using Construction \ref{mouf quadr}.

Let $J$ be a reduced spin factor of an anisotropic, non-degenerate quadratic space $(k,V,q)$ with base point $u$, let $X=C_1\otimes_k C_2$ and let the skew-symmetric form $\lsk.,.\rsk$ be as in Section~\ref{sec:E6E7E8}. Since quadrangular algebras of type $E_6,E_7$ and $E_8$ are determined by the similarity class of their there quadratic space. We have that the following maps coincide with the maps defined in \cite[Chapter 13]{TW}: 
\[a\cdot v=v\bt(u\bt a), \quad h(a,b)=\lsk u\bt a,b\rsk, \quad g(a,b)e_0=\half \lsk b,a\rsk.\]
Now define
\begin{compactitem}
\item the abelian group $V=(J_\subhalf,+)$;
\item the group $W=X_0\times J_0\cong X_0\times k$ with addition $[a_1,t_1e_0]\boxplus[a_2,t_2e_0]=[a_1+a_2,t_1e_0+t_2e_0+\half \lsk a_2,a_1\rsk]$. When we identify $J_0 \cong k: t e_0\leftrightarrow t$, we get \[[a_1,t_1]\boxplus[a_2,t_2]=[a_1+a_2,t_1+t_2+g( a_1,a_2)].\]
We will write $x_1(a,t):=x_1(a,te_0)$ and $x_3(a,t):=x_3(a,te_0)$.
\end{compactitem}
Let $U_1$ and $U_3$ be parametrized by $W$ and $U_2$ and $U_4$ be parametrized by $V$. Let $t,t_1,t_2\in k$, $v,v_1,v_2\in V$; we find the following commutator relations:
\begin{align*}
	[x_1(a_1,t_1), x_3(a_2, t_2)^{-1}] &= x_2\bigl(\lsk u\bt a_1, a_2 \rsk \bigr)=x_2\bigl(h( a_1,a_2) \bigr) \ , \\
	[x_2(v_1), x_4(v_2)^{-1}] &= x_3\bigl(0,f(v_1,v_2)e_0\bigr)= x_3\bigl(0,f(v_1,v_2)\bigr) \ , \\
	[x_1(a,t), x_4(v)^{-1}] &= [x_1(a,t e_0), x_4(v)^{-1}] \\
	&=x_2\bigl(\half \lsk u\bt a,a\cdot v\rsk +2 (t e_1)v\bigr) x_3\bigl(a\cdot v,U_vt e_1\bigr)\\
	&= x_2\bigl(\theta(a,v)+t v\bigr) x_3\bigl(a\cdot v,q(v)t \bigr)\ , \\
	[U_i, U_{i+1}] &= 1 \quad \forall i \in \{ 1, 2, 3 \} \ .
\end{align*}

We obtain exactly the  same description as in \cite[Example 16.6]{TW}. The Tits indices of the corresponding linear algebraic groups are as follows.
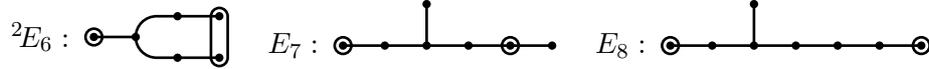
\[
\begin{tikzpicture}[line width=1pt, scale=.55]
	\node[left=5pt] (0,0) {$^2\!E_6:$};
	\draw (0,0) -- (1,0);
	\draw (1,0) arc (180:90:.5) -- (3,.5);
	\draw (1,0) arc (180:270:.5) -- (3,-.5);
	\diagnode{(0,0)}
	\diagnode{(1,0)}
	\diagnode{(2,.5)} \diagnode{(2,-.5)}
	\diagnode{(3,.5)} \diagnode{(3,-.5)}
	\distorbit{(0,0)}
	\distlongorbit{3}
\end{tikzpicture}
\quad
\begin{tikzpicture}[line width=1pt, scale=.55]
	\node[left=5pt] (0,0) {$E_7:$};
	\draw (0,0) -- (5,0);
	\draw (2,0) -- (2,1);
	\diagnode{(0,0)}
	\diagnode{(1,0)}
	\diagnode{(2,0)} \diagnode{(2,1)}
	\diagnode{(3,0)}
	\diagnode{(4,0)}
	\diagnode{(5,0)}
	\distorbit{(0,0)}
	\distorbit{(4,0)}
\end{tikzpicture}
\quad
\begin{tikzpicture}[line width=1pt, scale=.55]
	\node[left=5pt] (0,0) {$E_8:$};
	\draw (0,0) -- (6,0);
	\draw (2,0) -- (2,1);
	\diagnode{(0,0)}
	\diagnode{(1,0)}
	\diagnode{(2,0)} \diagnode{(2,1)}
	\diagnode{(3,0)}
	\diagnode{(4,0)}
	\diagnode{(5,0)}
	\diagnode{(6,0)}
	\distorbit{(0,0)}
	\distorbit{(6,0)}
\end{tikzpicture}
\]
\paragraph{Moufang quadrangles of involutory type}

As in \cite[chapter 11]{TW}, we define an involutory set in characteristic different from 2, as a triple $(L, L_\sigma, \sigma)$, where
$L$ is a field or a skew-field, $\sigma$ is an involution of $L$, $ L_\sigma=\{\ell\in L\mid \ell^\si=\ell\}=\{\ell+\ell^\si\mid \ell \in L\}$. (Notice that the definition in characteristic 2 is much more involved.)

Let $J=\mathcal{H}(M_2(L),\si t)$ (see Definition \ref{def: matrix}) and let $X$ be the zero module. 
\begin{compactitem}
\item Define the abelian group $V=(J_\subhalf,+)\cong(L,+)$ with isomorphism 
\[\begin{bmatrix}0&\ell^\si\\ \ell &0\end{bmatrix}\leftrightarrow \ell.\]
We will write $x_2(\begin{bmatrix}0&\ell^\si\\ \ell &0\end{bmatrix})=x_2(\ell)$ and $x_4(\begin{bmatrix}0&\ell^\si\\ \ell &0\end{bmatrix})=x_4(\ell)$.
\item Define the group $W=X_0\times J_0=\{[0,\al e_0]|\al \in L_\si\}\cong L_\si$ with addition $[0,\al_1 e_0]\boxplus[0,\al_2 e_0]=[0,(\al_1+\al_2)e_0]$. Therefore $W$ is isomorphic to the additive group of $L_\si$, we use the isomorphism $W\cong L_\si:[0,\al e_0]\leftrightarrow \al$. We will write $x_1(0,\al e_0)=x_1(\al)$ and $x_3(0,\al e_0)=x_3(\al)$.
\end{compactitem}
Let $U_1$ and $U_3$ be parametrized by $W$ and $U_2$ and $U_4$ be parametrized by $V$. Let $\al,\al_1,\al_2\in L_\si$ and consider the following elements of $V$
\[v=\begin{bmatrix}0&\ell^\si\\ \ell&0\end{bmatrix}, v_1=\begin{bmatrix}0&\ell_1^\si\\ \ell_1 &0\end{bmatrix}, v_2=\begin{bmatrix}0&\ell_2^\si \\ \ell_2 &0\end{bmatrix}\in J_\subhalf,\] using the formulas for the multiplication and the $U$-operator in $\mathcal{H}(M_2(L),\si t)$, we find for the commutator relations:
\begin{align*}
	[x_1(\al_1), x_3(\al_2)^{-1}] &= [x_1(0,\al_1e_0), x_3(0,\al_2 e_0)^{-1}]=x_2\bigl(0)=1 \ , \\
	[x_2(\ell_1), x_4(\ell_2)^{-1}] &= [x_2(v_1), x_4(v_2)^{-1}] =x_3\bigl(0,((\ell_1^\si\ell_2+\ell_2^\si\ell_1) e_0+(\ell_1\ell_2^\si+\ell_2 \ell_1^\si)e_1)e_0\bigr)\\
	&=x_3\bigl(0,(\ell_1^\si\ell_2+\ell_2^\si\ell_1)e_0\bigr)=x_3\bigl(\ell_1^\si\ell_2+\ell_2^\si\ell_1\bigr) \ , \\
	[x_1(\al), x_4(\ell)^{-1}] &= [x_1(0,\ell e_0), x_4(v)^{-1}] = x_2\bigl(2 (U_u\al e_0)v\bigr) x_3\bigl(0,U_vU_u\al e_0\bigr)\\
	&=  x_2\bigl(2 (\al e_1)v\bigr) x_3\bigl(0,U_v\al e_1\bigr)= x_2\left(   \begin{bmatrix}0&\ell^\si \al\\ \al\ell &0\end{bmatrix}\right) x_3\bigl(0,\ell^\si \al\ell e_0\bigr)\\
	&=x_2\bigl(\al\ell \bigr) x_3\bigl(\ell^\si \al\ell \bigr)\ , \\
	[U_i, U_{i+1}] &= 1 \quad \forall i \in \{ 1, 2, 3 \} \ .
\end{align*}
This is exactly \cite[Example 16.2]{TW}.
If $L$ is finite-dimensional over its center, of degree $d$, then these Moufang quadrangles arise from algebraic groups;
they are outer forms of $A_{4d-1}$ if the involution is of the second kind,
and they are (inner or outer forms) of absolute type $D_{2d}$ if the involution is of the first kind.

\paragraph{Moufang quadrangles of pseudo-quadratic type}

These are obtained in similar fashion as the quadrangular algebras in Section~\ref{sec:pseudo-quad}, but here we start from an arbitrary skew field with involution instead of starting from a quadratic pair. We repeat part of the setup from Section~\ref{sec:pseudo-quad}.

Let $L$ be a skew-field with involution $\si$. Let $(L,\si,X,h,\pi)$ be a standard pseudo-quadratic space (see Section~\ref{def:pqs}), so $\pi(a)=\half h(a,a)$ for all $a\in X$.

Let $J=\mathcal{H}(M_2(L),\si t)$ (see Definition \ref{def: matrix}) and let $\widetilde{X}=X^2$, the $1\times 2$ row vectors over X. For the action of $J$ on $\Xt$, for $j\in J, a\in \widetilde{X}$ we have $j\bt a=aj\in \widetilde{X}$.

As before we define $e_0=\begin{bmatrix}1&0\\0&0\end{bmatrix}, e_1=\begin{bmatrix}0&0\\0&1\end{bmatrix}, u=\begin{bmatrix}0&1\\1&0\end{bmatrix}\in J$. We have $\Xt_0=\{[a,0]\mid a\in X\}, \Xt_1=\{[0,a]\mid a\in X\}. $

We define the skew product $\Xt\times \Xt\rightarrow J$ as \[\lsk[a_1,a_2],[b_1,b_2]\rsk=\begin{bmatrix} h(a_1,b_1)-h(b_1,a_1)& h(a_1,b_2)-h(b_1,a_2)\\ -h(b_2,a_1)+h(a_2,b_1)& h(a_2,b_2)-h(b_2,a_2) \end{bmatrix}.\] 

\begin{compactitem}
\item Define the abelian group $V=(J_\subhalf,+)\cong(L,+)$ with isomorphism 
\[\begin{bmatrix}0&\ell^\si\\ \ell &0\end{bmatrix}\leftrightarrow \ell.\]
We will write $x_2(\begin{bmatrix}0&\ell^\si\\ \ell &0\end{bmatrix})=x_2(\ell)$ and $x_4(\begin{bmatrix}0&\ell^\si\\ \ell &0\end{bmatrix})=x_4(\ell)$.
\item Define the group $W=\Xt_0\times J_0\cong X\times L_\si$, when we identify $J_0 \cong L_\si: \al e_0\leftrightarrow \al$ and is  $\widetilde{X}_0\cong X:[a,0]\leftrightarrow a$,\footnote{ We have to denote elements of $\widetilde{X}$ and of $W$ both by $[.,.]$, from now on we only will make use of elements in $\Xt_0$ and not those contained in $X$ in general. We use the notation $[.,.]$ exclusively for elements of $W$ from now on.}  we get the addition \[[a_1,\al_1]\boxplus[a_2,\al_2]=[a_1+a_2,\al_1+\al_2+\half (h(a_2,a_1)-h(a_1,a_2))].\]
We will write $x_1(a,\al):=x_1([a,0],\al e_0)$ and $x_3(a,\al):=x_3([a,0],\al e_0)$. 
\end{compactitem}
For the commutator relations we obtain
\begin{align*}
	[x_1(a_1,\al_1), x_3(a_2,\al_2)^{-1}] &= x_2\bigl(h(a_1,a_2)\bigr) \ , \\
	[x_2(\ell_1), x_4(\ell_2)^{-1}] &=x_3\bigl(0,\ell_1^\si\ell_2+\ell_2^\si\ell_1) \ , \\
	[x_1(a,\al), x_4(\ell)^{-1}] &= x_2\bigl(\theta(a,\ell) + \al\ell \bigr) x_3\bigl(a\ell,\ell^\si \al\ell\bigr) \ , \\
	[U_i, U_{i+1}] &= 1 \quad \forall i \in \{ 1, 2, 3 \} \ ,
\end{align*}
for all $[a,t], [a_1,t_1], [a_2,t_2] \in W$ and all $\ell, \ell_1, \ell_2 \in L$.

In \cite[Example 16.5]{TW} $U_1$ and $U_3$ are parametrized  by the subset $T=\{[a,t]\mid \exists \al\in L_\si: \pi(a)=t+\al\}\subset X\times L$. We consider the bijection
\[\phi: T\rightarrow X\times L_\si:[a,t]\mapsto [a,-\pi(a)+t].\]
When we translate the group law and commutator relations in \cite[Example 16.5]{TW} from $T$ to $X\times L_\si$ using $\phi$,  we indeed obtain the expressions written above.

If $L$ is finite-dimensional over its center, of degree $d$, and $X$ is finite-dimensional over $L$,
then these Moufang quadrangles arise from algebraic groups.
If the involution is of the second kind, they are outer forms of absolute type $A_{\ell}$.
If the involution is of the first kind, they are of absolute type $C_{\ell}$ or $D_{\ell}$.

\vspace*{1ex}
\hrule
\vspace*{2ex}

\footnotesize

\noindent
Lien Boelaert,
Department of Mathematics,
Ghent University \\
Krijgslaan 281, S22,
B-9000 Gent, Belgium \\
{\tt lboelaer@cage.UGent.be}

\vspace{2ex}

\noindent
Tom De Medts,
Department of Mathematics,
Ghent University \\
Krijgslaan 281, S22,
B-9000 Gent, Belgium \\
{\tt tdemedts@cage.UGent.be}

\end{document}